\documentclass[reqno]{amsart}
\usepackage[english]{babel}
\usepackage{amsmath,amscd,amssymb,amsthm,amsxtra}
\usepackage{microtype}
\usepackage{mathrsfs}
\usepackage{accents}
\usepackage{color,graphicx,esint}
\usepackage{epsfig,epstopdf}
\usepackage[margin=1.2in]{geometry}

\theoremstyle{plain}
\newtheorem{theorem}{Theorem}[section]
\newtheorem{corollary}[theorem]{Corollary}
\newtheorem{proposition}[theorem]{Proposition}
\newtheorem{lemma}[theorem]{Lemma}

\theoremstyle{definition}
\newtheorem{definition}[theorem]{Definition}
\newtheorem{remark}[theorem]{Remark}

\theoremstyle{remark}
\numberwithin{theorem}{section}
\numberwithin{equation}{section}
\numberwithin{figure}{section}

\def\N{\mathbb{N}}

\def\R{\mathbb{R}}

\def\1{{\bf 1}}
\def\e{\mathrm{e}}
\def\d{\mathrm{d}}

\def\a{\alpha}
\def\b{\beta}

\def\g{\gamma}
\def\G{\Gamma}
\def\k{\kappa}
\def\l{\lambda}

\def\nab{\nabla}
\def\om{\omega}

\def\sig{\sigma}

\def\vep{\varepsilon}
\def\vphi{\varphi}

\def\bLa{\bar{L}_a}
\def\La{L_a}
\def\pa{\partial}
\def\na{\nabla}

\def\P{\mathcal{P}}

\DeclareMathOperator{\DIV}{div}

\DeclareMathOperator{\Span}{span}

\begin{document}

\title[Higher Free Boundary Regularity in the Fractional Obstacle Problem]{Higher Regularity of the Free Boundary in the Obstacle Problem for the Fractional Laplacian}

\author[Y.\ Jhaveri]{Yash Jhaveri}
\address{Department of Mathematics, The University of Texas at Austin,
Austin, TX 78712, USA}
\email{yjhaveri@math.utexas.edu}

\author[R.\ Neumayer]{Robin Neumayer}
\address{Department of Mathematics, The University of Texas at Austin,
Austin, TX 78712, USA}
\email{rneumayer@math.utexas.edu}

%~~~ABSTRACT~~~%
\begin{abstract}
We prove a higher regularity result for the free boundary in the obstacle problem for the fractional Laplacian via a higher order boundary Harnack estimate.
% \vspace{3mm}
% \noindent {\bf Keywords:} boundary Harnack, fractional Laplacian, free bondary, Obstacle problem, Schauder estimates
\end{abstract}
\maketitle
\vspace{-0.5cm}
%~~~MATH~~~%

\section{Introduction and Main Results}

In this paper, we investigate the higher regularity of the free boundary in the fractional obstacle problem. 
We prove a higher order boundary Harnack estimate, building on ideas developed by De Silva and Savin in \cite{DS3, DS4, DS5}. 
As a consequence, we show that if the obstacle is $C^{m,\beta}$, then the free boundary is $C^{m-1, \alpha}$ near regular points for some $0 < \a \leq \b$.
In particular, smooth obstacles yield smooth free boundaries near regular points.

\subsection{The Fractional Obstacle Problem.}
For a given function (obstacle) $\vphi \in C(\R^n)$ decaying rapidly at infinity and $s\in (0,1)$, a function $v$ is a solution of the fractional obstacle problem if
\begin{equation}
\label{eqn: obstacle problem}
\begin{cases}
v(x) \geq \vphi(x)  &\text{in } \R^n\\
\lim_{|x| \to \infty} v(x) = 0 &\text{on }\R^n \\
(-\Delta)^{s} v(x) \geq 0 &\text{in } \R^n\\
(-\Delta)^{s} v(x) = 0 &\text{in } \{ v > \vphi \}
\end{cases}
\end{equation}
where the $s$-Laplacian $(-\Delta)^{s}$ of a function $u$ is defined by
\[
(-\Delta)^{s} u(x) := c_{n,s} \, {\rm PV} \int_{\R^n} \frac{u(x) - u(x+z)}{|z|^{n+2s}} \, \d z.
\]
The sets 
\[
\P:= \{ v  = \vphi \} \qquad \text{and} \qquad \G := \pa \{ v  = \vphi \}
\]
are known as the {\it contact set} and the {\it free boundary} respectively.

The fractional obstacle problem appears in many contexts, including the pricing of American options with jump processes (see \cite{CT} and the Appendix of \cite{BFRO} for an informal discussion)
and the study of the regularity of minimizers of nonlocal interaction energies in kinetic equations (see \cite{CDM}).

While the obstacle problem for the fractional Laplacian is nonlocal, it admits a local formulation thanks to the extension method (see \cite{CS, MO}).
Specifically, one considers the $a$-harmonic\footnote{\,We say a function $u$ is \textit{$a$-harmonic} if $L_a u =0$.} extension $\tilde{v}$ of $v$ to the upper half-space $\R^{n+1}_+ := \R^n \times (0,\infty)$:
\[
\begin{cases}
\La \tilde{v}(x,y) = 0 &\text{in } \R^{n+1}_+ \\
\tilde{v}(x,0) = v(x) &\text{on } \R^n
\end{cases}
\] 
where
\[
\La u(x,y) := \DIV(|y|^a \nab u(x,y)) \qquad\text{and}\qquad a: = 1-2s \in (-1,1).
\]
The function $\tilde{v}$ is obtained as the minimizer of the variational problem
\[
\min\bigg\{ \int_{\R^{n+1}_+} |\nab u|^2 \, |y|^a \d x \, \d y \,:\, u \in H^1(\R^{n+1}_+, |y|^a), \, u(x,0) = v(x) \bigg\}
\]
and satisfies
\[
-\lim_{y \to 0} |y|^a\tilde{v}(x,y) = (-\Delta)^s v(x) \qquad\forall \, x \in\R^n.
\]
After an even reflection across the hyperplane $\{ y = 0 \}$, \eqref{eqn: obstacle problem} is equivalent to the following local problem. 
For a given $a\in (-1,1)$ and a function $\vphi \in C(\R^n)$ decaying rapidly at infinity, a function $\tilde{v}$ is a solution of the localized fractional obstacle problem if it is even in $y$ and satisfies
\begin{equation}
\label{eqn: obstacle problem local}
\begin{cases}
\tilde{v}(x,0) \geq \vphi(x) &\text{on } \R^n\\
\lim_{|(x,y)| \to \infty} \tilde{v}(x,y) = 0 &\text{on } \R^{n+1} \\
\La \tilde{v}(x,y) \leq 0 &\text{in } \R^{n+1} \\
\La \tilde{v}(x,y) = 0 &\text{in } \R^{n+1} \setminus \{ \tilde{v}(x,0) = \vphi(x) \}.
\end{cases}
\end{equation}
When $a = 0$, i.e., $s = 1/2$, the operator $\La$ is the Laplacian, and \eqref{eqn: obstacle problem local} is the well-known Signorini (thin obstacle) problem, which can be stated not only in all of $\R^{n+1}$, but in suitable bounded domains of $\R^{n+1}$.
For example, a typical formulation of the Signorini problem is in $B_1 \subset \R^{n+1}$:
\begin{equation}
\label{eqn: signorini}
\begin{cases}
\tilde{v}(x,0) \geq \vphi(x) &\text{on } B_1 \cap \{ y = 0 \}\\
\tilde{v}(x,y) = g(x,y) &\text{on } \pa B_1 \\
\Delta \tilde{v}(x,y) \leq 0 &\text{in } B_1 \\
\Delta \tilde{v}(x,y) = 0 &\text{in } B_1 \setminus \{ \tilde{v}(x,0) = \vphi(x) \}.
\end{cases}
\end{equation}

Primary questions in obstacle problems are the regularity of the solution and the structure and regularity of the free boundary. 
The local formulation of the fractional obstacle problem, \eqref{eqn: obstacle problem local}, allows the use of local PDE techniques to study the regularity of the solution and the free boundary. 
Under mild conditions on the obstacle\footnote{\,In \cite{CSS}, the obstacle is assumed to be $C^{2,1}$, though in \cite{CDS}, this is relaxed, and $\vphi$ is only assumed to be $C^{1,s+\vep}$.}, Caffarelli, Salsa, and Silvestre show, in \cite{CSS}, that the solution of \eqref{eqn: obstacle problem} is optimally $C^{1,s}$ using the almost-optimal regularity of the solution shown via potential theoretic techniques in \cite{S}. 
Furthermore, studying limits of appropriate rescalings of the solution (blowups) at points on the free boundary, they show the blowup at $x_0 \in \G$ must either have (a) homogeneity $1+s$ or (b) homogeneity at least $2$. 
The points at which the blowup is $(1+s)$-homogeneous are known as {\it regular points} of $\G$.
In \cite{CSS}, they show that the regular points form a relatively open subset of $\G$ and that $\G$ is $C^{1,\sig}$ near these regular points for some $0<\sig<1$. 
For the case $s=1/2$, analogous results were first shown in \cite{AC, ACS}. 
The structure of the free boundary away from regular points is investigated, for example, in \cite{BFRO} and \cite{GP}.

\subsection{Main Result and Current Literature.}
Our main result is the following: 

\begin{theorem}
\label{thm: main}
Let $\vphi \in C^{m,\beta}(\R^n)$ with $m \geq 4$ and $\b \in (0,1)$ or $m=3$ and $ \beta =1$.
Suppose $x_0$ is a regular point of the free boundary $\G = \G(v)$ of the solution $v$ to \eqref{eqn: obstacle problem}.
Then, $\G \in C^{m-1,\a}$ in a neighborhood of $x_0$ for some $\a \in (0,1)$ depending on $s,n,m$, and $\beta$.
In particular, if $\vphi \in C^\infty(\R^n)$, then $\G \in C^\infty$ near regular points.
\end{theorem}

Starting from the $C^{1,\sig}$ regularity obtained in \cite{CSS}, the H\"{o}lder exponent $\a$ obtained in Theorem~\ref{thm: main} is the minimum of $\b$ and $\sig$.  
In order to prove Theorem~\ref{thm: main}, we establish a higher order boundary Harnack estimate for the operator $L_a$. We prove this estimate in slit domains, that is, domains in $\R^{n+1}$ from which an $n$-dimensional slit $\P \subset \{ x_{n+1} = 0\}$ is removed. Recall that the classical boundary Harnack principle states that the quotient of two positive harmonic functions that vanish continuously on a portion of the boundary of a Lipschitz domain is H\"{o}lder continuous up to the boundary (see \cite{CFMS,HW}).
In \cite{DS4}, De Silva and Savin remarkably extend this idea to the {\it higher order boundary Harnack principle}, proving that the quotient of two positive harmonic functions that vanish continuously on a portion of the boundary of a $C^{k,\a}$ domain is $C^{k,\a}$. Motivated by applications to the Signorini problem, in \cite{DS5}, they
prove such a higher order boundary Harnack principle on slit domains. To do so, they assume $\G := \pa_{\R^n} \P$ is locally the graph of a function of the first $n-1$ variables, and they consider a coordinate system $(x,r)$ on $\R^{n+1}$ where $x \in \R^n$ and $r$ is the distance to $\G$.
They also define a corresponding notion of H\"{o}lder regularity $C^{k,\a}_{x,r}$ (see Section~\ref{sec: Prelim}) that restricts to the standard notion of $C^{k,\a}$ on $\G$ when $\G \in C^{k,\a}$.

\begin{theorem}[De Silva and Savin, \cite{DS5}]
\label{thm: DS5 HOBH}
Let $u$ and $U>0$ be harmonic functions in $B_1 \setminus \P \subset \R^{n+1}$ that are even in $x_{n+1}$ and vanish continuously on the slit $\P$. 
Suppose $ 0 \in \G := \partial_{\R^n} \P $ with $\G \in C^{k,\a}$ for $k \geq 1$ and $\| \G \|_{C^{k,\a}}\leq 1$. If  $\|u\|_{L^\infty(B_1)} \leq 1$ and $U(\nu(0)/2) = 1$, where $\nu$ is the outer unit normal to $\P$,
then
\[
\bigg\|\frac{u}{U}\bigg\|_{C^{k,a}_{x,r}(\G \cap B_{1/2})} \leq C
\]
for some $C = C(n,k,\a) > 0$.
\end{theorem}
Here, $\| \G \|_{C^{k,\a}}$ is defined as in \eqref{eqn: gamma norm}.
Theorem~\ref{thm: DS5 HOBH} is used to prove that the free boundary is smooth near regular points for \eqref{eqn: signorini} with $\vphi \equiv 0$ in the same way that the boundary Harnack principle is used to improve the regularity of the free boundary near regular points from Lipschitz to $C^{1,\a}$ in, for example, the classical obstacle problem.
Indeed, if $\G$ is locally the graph of a $C^{k,\a}$ function $\g$ of the first $n-1$ variables, then Theorem~\ref{thm: DS5 HOBH} implies that $\pa_i \tilde{v}/\pa_n \tilde{v}$ is $C^{k,\a}$ on $\G$ where $\tilde{v}$ is the solution to the Signorini problem.
On the other hand, $\G$ is also the zero level set of $\tilde{v}$, and so $\pa_i \g$, for each $i =1,\dots, n-1$, is given by $\pa_i \tilde{v}/\pa_n \tilde{v}$ on $\G$.
Hence, $\g \in C^{k+1,\a}$.
Starting with $k = 1$ and proceeding iteratively, one proves that the free boundary is smooth near regular points.

The proof of the higher order boundary Harnack estimate in this paper is motivated by the global strategy developed by De Silva and Savin to prove Theorem~\ref{thm: DS5 HOBH}. However, some delicate arguments are needed to adapt these ideas to our setting, which we briefly describe here. The proof involves a perturbative argument,
the core regularity result being one in which $\G$ is flat (Proposition~\ref{prop: Schauder flat boundary}). When $a = 0$, the flat case follows from capitalizing on the structural symmetry of the Laplacian; boundary regularity is inherited from interior regularity for a reflection of the solution.
Instead, to handle the case a $\neq 0$, we prove the necessary regularity of our solutions by hand.
First, a reduction argument allows us to focus on the two-dimensional case. 
Here, the specific degeneracy of the operator $\La$, for each $a \neq 0$, forces a specific degeneracy in solutions that vanish on the negative $x$-axis to the equation $\La u = 0$.
We observe this first in global homogeneous solutions.
Then, we prove that our solutions have a well-defined power series-like decomposition in terms of these homogeneous solutions, which in turn yields the regularity result.
(See Section~\ref{sec: Flat}.)

Another new difficulty we encounter comes from considering nonzero obstacles in \eqref{eqn: obstacle problem}.
As discussed above, Theorem~\ref{thm: DS5 HOBH} implies $C^{\infty}$ regularity of the free boundary near regular points in the Signorini problem with zero obstacle, that is, taking $\vphi \equiv 0$ in \eqref{eqn: signorini}.
Yet, taking $\vphi \equiv 0$ in the nonlocal problem \eqref{eqn: obstacle problem} is rather uninteresting: the solution is identically zero.
While the results of \cite{DS5} do extend to \eqref{eqn: obstacle problem} when $s = 1/2$ so long as an extension of $\vphi$ can be subtracted off without producing a right-hand side\footnote{\,For instance, this can be done if $\Delta^m \vphi = 0$ for some $m \in \N$.}, such an extension is not generally possible.  
The new feature of the higher order boundary Harnack estimates here, necessary for treating general obstacles, is that we allow both $\La u$ and $\La U$ to be nonzero when considering the quotient $u/U$.
Still, handling general obstacles in \eqref{eqn: obstacle problem} is quite involved. 
Even after constructing a suitable extension of the obstacle from $\R^n$ to $\R^{n+1}$, one finds a large gap between having and applying these propositions, another consequence of having to work with a degenerate elliptic operator.
(See Section~\ref{sec: Proof of Main Thm}.)

Proposition~\ref{prop: HOBH k pos} and Proposition~\ref{prop: HOBH 1 pos} are the higher order boundary Harnack estimates in the generality needed to prove Theorem~\ref{thm: main}. 
The simplest case of these estimates, however, takes the following form.

\begin{theorem}[Higher Order Boundary Harnack Estimate]
\label{thm: La HOBH}
Let $u$ and $U>0$ be $a$-harmonic functions in $B_1 \setminus \P \subset \R^{n+1}$ that are even in $x_{n+1}$ and vanish continuously on the slit $\P$. 
Suppose $ 0 \in \G := \partial_{\R^n} \P $ with $\G \in C^{k,\a}$ for $k \geq 1$ and $\| \G \|_{C^{k,\a}}\leq 1$. If  $\|u\|_{L^\infty(B_1)} \leq 1$ and $U(\nu(0)/2) = 1$, where $\nu$ is the outer unit normal to $\P$,
then
\[
\bigg\|\frac{u}{U}\bigg\|_{C^{k,a}_{x,r}(\G \cap B_{1/2})} \leq C
\]
for some $C = C(a,n,k,\a) > 0$.
\end{theorem}

The original approach to proving higher regularity in obstacle problems, pioneered in \cite{KN}, was to use the hodograph-Legendre transform.
These techniques have been extended to prove higher regularity in the Signorini problem with zero obstacle in \cite{KPS} and in thin obstacle problems with variable coefficients and inhomogeneities in \cite{KRS1}. 
On the other hand,  at the same time as \cite{KPS}, De Silva and Savin used the higher order boundary Harnack principle to show higher free boundary regularity in the Signorini problem with zero obstacle in \cite{DS5}, as we discussed above. They also used these techniques to give a new proof of higher regularity of the free boundary in the classical obstacle problem (see \cite{DS4}).
The higher order boundary Harnack approach has been adapted to the parabolic setting, proving higher regularity of the free boundary for the parabolic obstacle problem in \cite{BG} and for the parabolic Signorini problem with zero obstacle in \cite{BGZ}.
We mention that an advantage of the hodograph-Legendre transform approach is that it allows one to prove analyticity of the free boundary near regular points when the obstacle is analytic.
\\

Upon completion of this work, we learned that Koch, R\"{u}land, and Shi in \cite{KRS2} -- independently and at the same time -- have proven an analogous result to our Theorem~\ref{thm: main}.
In contrast to the methods used herein, they employ the partial hodograph-Legendre transform techniques mentioned above.

\subsection{Organization}
In Section~\ref{sec: Prelim}, we introduce some notation and present some useful properties of the operator $\La$.
In Section~\ref{sec: HOBH}, we state and prove a pointwise higher order boundary Harnack estimate.
Section~\ref{sec: Schauder Cka} is dedicated to proving a pointwise Schauder estimate, while Section~\ref{sec: C1a} extends the results of the previous two sections to a lower regularity setting.
In Section~\ref{sec: Flat}, we prove a regularity result for $a$-harmonic functions vanishing continuously on a hyperplane.
Finally, we prove Theorem~\ref{thm: main} in Section~\ref{sec: Proof of Main Thm}. 
\\

\noindent{\bf Acknowledgments.} We wish to thank Alessio Figalli for suggesting this problem.
Part of this work was done while the authors were guests of the \'Ecole Normale Sup\'eriore de Lyon in the fall of 2015; the hospitality of ENS Lyon is gratefully acknowledged. 
RN is supported by the NSF Graduate Research Fellowship under Grant DGE-1110007.

%~~~~~~~~~~~~~~~~~~~~~~~~~~~~~~~~~~~~~~~~~~~~~~~~~~~~~~~%
\section{Preliminaries}
\label{sec: Prelim}

\subsection{Notation and Terminology}\label{subsec: Notation}

Let $X \in \R^{n+1}$ be given by
\[
X = (x', x_n, y) = (x',z) = (x, y)
\]
where $x' \in \R^{n-1}$, $x_n \in \R$ is the $n$th component of $X$, and $y \in \R$ is the $(n+1)$st component of $X$. 
In this way, $x = (x',x_n)\in \R^n$ and $z = (x_n,y) \in \R^2$.
Furthermore, define 
\[
B_\l := \{ |X| < \l \}, \qquad
B_\l^* := \{ |x| < \l \}, \qquad\text{and}\qquad
B_\l' := \{ |x'| < \l \}.
\]

Let $\G$ be a codimension two surface in $\R^{n+1}$ of class $C^{k+2,\alpha}$ with $k \geq -1$.
Then, up to translation, rotation, and dilation, $\G$ is locally given by 
\[
\G = \{ (x',\gamma(x'),0) :  x' \in B_1' \}
\]
where $\gamma : B'_1 \to \R$ is such that
\[
\gamma(0) = 0, \qquad \nab_{x'}\gamma(0) = 0, \qquad\text{and}\qquad \|\gamma\|_{C^{k+2,\alpha}(B'_1)} \leq 1.
\]
We let 
\begin{equation}
\label{eqn: gamma norm}
\| \G\|_{C^{k,\a} }: = \| \g\|_{C^{k,\a}(B_1')}.
\end{equation}

Define the $n$-dimensional {\it slit} $\mathcal{P}$ by
\[
\mathcal{P} := \{ x_n \leq \gamma(x'),\, y = 0 \},
\]
and notice that $\partial_{\R^n} \mathcal{P} = \G$.

Let $d = d(x)$ denote the signed distance in $\R^n$ from $x$ to $\G$ with $d > 0$ in the $\e_n$-direction, and let $r = r(X)$ be the distance to $\G$ in $\R^{n+1}$:
\[
r := (y^2 + d^2)^{1/2}.
\]
Then,
\[
\nab_x r = \frac{d}{r}\nu,  \qquad \nab_x d = \nu, \qquad\text{and}\qquad 
-\Delta_x d=\k 
\]
where $\k = \k(x)$ denotes the mean curvature and $\nu = \nu(x)$ represents the unit normal in $\R^n$ of the parallel surface to $\G$ passing through $x$.
Moreover, set
\begin{equation}\label{eqn: Ua def}
U_a := \bigg(\frac{r+d}{2}\bigg)^s.
\end{equation}
Observe that one can express $U_a$ as
\begin{equation}
\label{eqn: Ua property}
U_a = \frac{|y|^{2s}}{2^{s}(r-d)^s},
\end{equation}
and when $\G$ is flat, that is, $\g \equiv 0$, $U_a$ is equal to
\begin{equation}\label{eqn: Uabar def}
\bar{U}_a:= \bigg(\frac{|z| + d}{2}\bigg)^s.	
\end{equation}

As shown in \cite{CSS}, $\bar{U}_a$ is (up to multiplication by a constant) the only positive $a$-harmonic function vanishing on $\{x_n \leq 0,\,  y=0 \}$.
Thus, if $\tilde{v}$ is a global homogeneous solution of \eqref{eqn: obstacle problem local} with $\vphi \equiv 0$, then, up to a rotation, $\bar{U}_a = \pa_\nu \tilde{v}$, where $\nu$ is the unit normal in $\R^n$ to the free boundary $\G$. 
When $\G$ is not flat, the function $U_a$ is not $a$-harmonic.
However, it approximates $a$-harmonic functions in the sense of the Schauder estimates of Proposition~\ref{prop: pointwise Schauder Cka}.

We work in the coordinate system determined by $x$ and $r$.
For a polynomial 
\[
P = P(x,r) = p_{\mu m} x^{\mu}r^{m},
\]
we let 
\[
\| P \|  := \max \{ |p_{\mu m}| \}.
\]
Here, $\mu$ is a multi-index, $|\mu| = \mu_1 + \cdots + \mu_n$, $\mu_i \geq 0$, and $x^\mu = x_1^{\mu_1}\cdots x_n^{\mu_n}$.
It is useful to think that the coefficients $p_{\mu m}$ are defined for all indices $(\mu,m)$ by extending by zero.
Frequently, we use the convention of summation over repeated indices.

A function $f: B_1 \to \R$ is {\it pointwise $C^{k,\a}_{x,r}$ at $X_0 \in \G$} if there exists a (tangent) polynomial $P(x,r)$ of degree $k$ such that 
\[
f(X) = P(x,r) + O(|X-X_0|^{k+\a}).
\]
We will write $f \in C^{k,\a}_{x,r}(X_0)$, and $\|f\|_{C^{k,\a}_{x,r}(X_0)}$ will denote the smallest constant $M >0$ for which
\[
\|P\| \leq M \qquad\text{and}\qquad |f(X)-P(x,r)| \leq M|X-X_0|^{k+\a}.
\]

We will call objects {\it universal} if they depend only on any or all of $a,n,k$, or $\alpha$.
Throughout, unless otherwise stated, $C$ and $c$ will denote positive universal constants that may change from line to line.

\subsection{Basic Regularity Results for $L_a$}
Let us collect some regularity results for weak solutions $u \in H^1(B_\l, |y|^a)$ of the equation $L_a u = |y|^a f$, beginning with interior regularity. 
Throughout the section, we assume that $0<\l \leq 1.$

As the weight $|y|^a$ is an $A_2$-Muckenhoupt weight, we obtain the following local boundedness property for subsolutions $L_a$ from \cite{FKS}.

\begin{proposition}[Local Boundedness]
\label{prop: harnack}
Let $u \in H^1(B_{\l}, |y|^a)$ be such that $\La u \geq |y|^a f$ in $B_\l$ with $f \in L^{\infty}(B_\l)$.
Then,
\begin{equation}\label{eqn: weak Harnack}
\sup_{B_{\l/2}} u \leq  C\bigg(\frac{1}{\l^{n+1+a}} \int_{B_{\l}} |u|^2 \, |y|^a \, \d X \bigg)^{1/2} +C\l^2 \|f\|_{L^{\infty}(B_\l)}.
\end{equation}
\end{proposition}

\begin{proof}
The inequality follows from \cite[Theorem 2.3.1]{FKS} applied to $u = c(\l^2-|X|^2)$ and $u + c|X|^2$; see \cite[Lemma 3.4]{BFRO} for details.\footnote{\,We caution the reader that the authors in \cite{BFRO} define $\La$ with the opposite sign convention, that is, they consider $\La u(x,y) := -\DIV(|y|^a \nab u(x,y))$.}
\end{proof}	

The operator $L_a$ also enjoys a Harnack inequality \cite[Lemma 2.3.5]{FKS} and a boundary Harnack inequality \cite[p. 585]{FKJ2}. 
By a standard argument (see, for instance, \cite[Theorem 8.22]{GT}), the Harnack inequality implies that solutions of $L_a u = |y|^a f$ for bounded $f$ are H\"{o}lder continuous.

\begin{proposition}[H\"{o}lder Continuity]
\label{cor: holder}
Let $u \in H^1(B_{\lambda}, |y|^a)$ be such that $\La u = |y|^a f$ in $B_\l$ with $f \in L^{\infty}(B_\l)$. 
Then, 
\[
\|u\|_{C^{0,\a}(B_{\l/2})} \leq  C\l^{-\a}\|u\|_{L^\infty(B_{\l})} +C\l^{2-\a} \|f\|_{L^\infty(B_\l)}
\]
for some $\a \in (0,1)$.
\end{proposition}

If $u$ is such that $L_a u = 0$, then $\La (\pa_i u) = 0$ for $i = 1,\dots,n$.
Moreover, as pointed out in \cite{CS}, the function $|y|^a\pa_y u$ satisfies $L_{-a} (|y|^a\pa_y u) = 0$.
If instead $L_a u = |y|^a f$ for $f$ bounded, then one can show that
\[
\La (\pa_i u) = |y|^a\pa_i f \qquad\forall i = 1,\dots,n \qquad \text{ and } \qquad L_{-a} (|y|^a\pa_y u) = \pa_y f.\]
Here, the partial derivatives of $f$ are understood in the distributional sense.
And so, we have the following regularity result for $\nab_x u$ and $|y|^a \pa_y u$.

\begin{proposition}[Interior Regularity of $\nab_x u$ and $|y|^a \pa_y u$]
\label{prop: derivative estimates}
Let $u \in H^1(B_{\l}, |y|^a)$ be such that $\La u = |y|^a f$ in $B_\l$ with $f \in L^{\infty}(B_\l)$.
Then,
\[
\|\na_x u\|_{L^\infty(B_{\l/4})} \leq C\l^{-1}\|u\|_{L^\infty(B_\l)} +C \|f\|_{L^\infty(B_\l)}.
\]
Furthermore, if $g:=|y|^a\pa_y f \in L^\infty(B_\l)$, then 
\[
\| |y|^a \pa_y u \|_{L^\infty(B_{\l/4})} \leq C \l^{a-1}\| u\|_{L^\infty(B_{\l})} +C  \l^{a}\|f\|_{L^\infty(B_\l)}+ C\l^2\| g\|_{L^\infty(B_{\l})}.
\]
\end{proposition}
In fact, $\pa_i u$ and $|y|^a \pa_y u $ are H\"{o}lder continuous, but boundedness is all we need.

\begin{proof}
%The H\"{o}lder continuity of $\nab_x u$ is established by \cite[Theorem~2.3.15]{FKS}.
%In particular,
%\begin{equation}\label{eqn: deriv holder}
%\|\na_x u\|_{C^{0,\a}(B_{\l/2})} \leq  C\|\na_x u\|_{L^\infty(B_\l)} +C\l^2 \|f\|_{L^\infty(B_\l)}.
%\end{equation}
By \cite[Theorems 2.3.1 and 2.3.14]{FKS}, $\na_x u$ has the following local boundedness property:
\[
\sup_{B_{\l/4}} |\na_x u| \leq 
C\bigg(\frac{1}{\l^{n+1+a}} \int_{B_{\l/2}} |\na_x u|^2 \, |y|^a \, \d X \bigg)^{1/2}+C\l^\delta \| f\|_{L^\infty(B_\l/2)} 
\]
for some $\delta>0$.
Using the energy inequality
\begin{equation}\label{eqn: energy inequality}
\int_{B_{\l/2}} |\na u|^2|y|^a \, \d X  \leq C \int_{B_\l} |f|^2|y|^a \, \d X + \frac{C}{\l^2} \int_{B_\l} |u|^2 |y|^a \, \d X
\end{equation}
for $u$ and recalling $\l\leq 1$, the first estimate follows.
%\[
%\sup_{B_{\l/2}} |\na_x u| \leq
% C \| f\|_{L^\infty(B_\l)}  + \frac{1}{\l} \| u\|_{L^\infty(B_\l)}.
% \] 
 
Let $w:= |y|^a \pa_y u$ and note that $L_{-a} w = |y|^{-a} g$. Since $g\in L^\infty(B_\l),$ \eqref{eqn: weak Harnack} implies that
\[
\begin{split}
\|w\|_{L^\infty(B_{\l/4})}  &\leq C\bigg( \frac{1}{\l^{n+1-a}}\int_{B_{\l/2} } |w|^2 |y|^{-a} \, \d X \bigg)^{1/2} + C\l^2\|g \|_{L^\infty(B_{\l/2})} \\
&\leq C\bigg( \frac{1}{\l^{n+1-a}}\int_{B_{\l/2} } |\na u|^2 |y|^{a} \, \d X \bigg)^{1/2} + C\l^2\|g \|_{L^\infty(B_{\l/2})}.
\end{split}
\] 
Applying \eqref{eqn: energy inequality} once more concludes the proof.
%we conclude that
%\[
%\| w \|_{L^\infty(B_{\l/4})} \leq C \l^{a-1}\| u\|_{L^\infty(B_{\l})} +C  \l^{a}\|f\|_{L^\infty(B_\l)}+ C\l^2\| g\|_{L^\infty(B_{\l/2})} .
% \]
\end{proof}

The following boundary regularity result is a consequence of Proposition~\ref{prop: derivative estimates} applied to odd reflections.
More specifically, let $B_\l^+ := B_\l \cap \{ y > 0 \}$ and $u \in H^1(B_\l^+, y^a)$ be such that
\begin{equation}
\label{eqn: Half Ball Dirichlet}
\begin{cases}
L_a u = y^a f &\text{in } B_\l^+\\
u=0 &\text{on } \{y=0\},
\end{cases} 
\end{equation}
and let $\bar u$ and $\bar f$ be the odd extensions across $\{ y=0\}$ of $u$ and $f$ respectively. 
Then, notice that $L_a \bar u = |y|^a \bar f$ in $ B_\l$. Applying Proposition~\ref{prop: derivative estimates} to $\bar u$ yields the following.

\begin{corollary}[Boundary Regularity for $\nab_x u$ and $y^a \pa_y u$]
\label{cor: basic bd reg}
Suppose $u\in H^1(B_\l^+, y^a)$ satisfies \eqref{eqn: Half Ball Dirichlet} where $f \in C(\overline{B}_\l^+)$.
Then,
\[
\|\na_x u\|_{L^\infty(B^+_{\l/4})} \leq C\l^{-1}\|u\|_{L^\infty(B_\l^+)} +C \|f\|_{L^\infty(B_\l^+)} .
\]
Furthermore, if $f$ vanishes on $\{y=0\}$ and $y^a\pa_y f \in L^\infty(B^+_\l)$, then letting $g := y^a \pa_y f$, 
\[
\| y^a \pa_y u \|_{L^\infty(B^+_{\l/4})} \leq C \l^{a-1}\| u\|_{L^\infty(B_{\l}^+)} +C  \l^{a}\|f\|_{L^\infty(B_\l^+)}+ C\l^2\| g\|_{L^\infty(B_{\l}^+)}.
\]
\end{corollary}

If $f$ does not vanish on $\{ y = 0\}$ yet depends only on $x$, then we have the following.

\begin{corollary}[Boundary Regularity of $y^a \pa_y u$]
\label{cor: y bd reg 2}
Suppose $u\in H^1(B_\l^+, y^a)$ satisfies \eqref{eqn: Half Ball Dirichlet}
where $f = f(x)$ and $f \in L^\infty(B_\l^+)$.
% and $y^a\pa_y f \in L^\infty(B_\l^+)$. 
Then, 
%letting $g:= y^a \pa_y f$, 
\[
\| y^a \pa_y u \|_{L^\infty(B^+_{\l/4})} \leq C \l^{a-1}\| u\|_{L^\infty(B^+_{\l})} +C\l^{a} \|f\|_{L^\infty(B^+_\l)}.
%+ C\l^2\| g\|_{L^\infty(B^+_{\l/2})} .
\]
\end{corollary}

\begin{proof}
Letting $w:= |y|^a \pa_y \bar u$, where again $\bar u$ is the odd extension of $u$ across $\{y=0\}$, we have
\[ 
L_{-a}w = 2f \mathcal{H}^{n}\llcorner\{y=0\}.
\]
Let $M:= \| f\|_{L^\infty(B_\l^+)}$ and consider the barriers
\[ 
v_{\pm} := w \pm \frac{M}{1+a}|y|^{a+1},
\]
which satisfy 
\[
L_{-a} v_{\pm} = ( 2f \pm 2M) \mathcal{H}^{n} \llcorner\{ y =0 \} .
\]
Therefore, $L_{-a} v_+ \geq 0 $ and $ L_{-a} v_- \leq 0$. 
Applying Proposition~\ref{prop: harnack} and arguing as in the proof of Proposition~\ref{prop: derivative estimates}, we see that
\[
\sup_{B_{\l /4}} v_+ \leq C\l^{a-1}\| u\|_{L^\infty(B_\l)}
+ C \l^a \| f\|_{L^\infty(B_\l)}.
\]
The same bound holds for $\sup_{B_{\l/4}} -v_-$. 
As $v_- \leq w \leq v_+$, this concludes the proof.
\end{proof}

As a consequence of Corollaries~\ref{cor: basic bd reg} and \ref{cor: y bd reg 2}, we have the following growth estimate on $|\nab_x u|$ when $f=f(x)$ is Lipschitz.

\begin{corollary}[Boundary Growth of $\pa_i u$]
\label{cor: upgrade f is C1}
Suppose $u\in H^1(B_\l^+, y^a)$ satisfies \eqref{eqn: Half Ball Dirichlet}
where $f = f(x)$ and $f \in C^{0,1}(\overline{B}_\l^+)$.
% and $y^a\pa_y f \in L^\infty(B_\l^+)$. 
Then, for any $i \in 1,\hdots, n$,
\[
|\pa_i u(X)| \leq C y^{2s} \quad \text{in } B_{\l/4}^+.
\] 
Here, the constant $C>0$ depends on $a, n, \l, \| u\|_{L^\infty(B^+_{\l})}$, and $ \|f\|_{C^{0,1}(\overline{B}^+_\l)}$.
\end{corollary}
\begin{proof}
For any $i = 1,\dots, n$, we have that $L_a	(\pa_i u) = y^a h$ where $h:=\pa_i f$.
Since $f\in C^{0,1}(\overline{B}_{\l}^+)$, it follows that $h \in L^\infty(B_\l^+)$. 
Applying Corollary~\ref{cor: y bd reg 2} to $\pa_i u$ implies that 
\[
|\pa_y(\pa_i u(X))| \leq Cy^{-a},
\]
where, using Proposition~\ref{prop: derivative estimates}, we see that $C$ depends on $a, n, \l, \| u\|_{L^\infty(B^+_{\l})}, \|f\|_{C^{0,1}(\overline{B}^+_\l)}$.
Since $\pa_i u(x, 0) = 0$, we determine that
\[
|\pa_i u(X)|  \leq C \bigg| \int_0^y t^{2s-1} \, \d t \bigg|  = C y^{2s}.
\] 
\end{proof}

We have the same growth estimate for $|\nab_x u|$ when $f$ is less regular and unconstrained to depend only on $x$ so long as it vanishes on $\{ y = 0 \}$.

\begin{corollary}[Boundary Growth of $\pa_i u$]
\label{cor: upgrade f is 0}
Suppose $u \in H^1(B_\l^+, y^a)$ satisfies \eqref{eqn: Half Ball Dirichlet} where $f \in C(\overline{B}_\l^+)$ and $f$ vanishes on $\{ y = 0 \}$.
If $g :=y^a \pa_y f \in L^\infty(B_\l^+)$, then for any $i \in 1,\hdots, n$,
\[
|\pa_i u(X)| \leq C y^{2s} \quad \text{in } B_{\l/4}^+.
\] 
Here, the constant $C>0$ depends on $a, n, \l, \| u\|_{L^\infty(B^+_{\l})}, \|f\|_{L^\infty(B^+_\l)}$, and $\|g\|_{L^{\infty}(B^+_\l)}$.
\end{corollary}

\begin{proof}
Let $\bar{w}:= |y|^a \pa_y \bar{u}$ and note that $L_{-a} \bar{w} = |y|^{-a}\bar{g}$ in $B_\l$ where $\bar{g} := |y|^a\pa_y\bar{f} \in L^\infty(B_\l)$.
From Proposition~\ref{prop: derivative estimates}, we find that 
\[
\|\pa_i \bar{w}\|_{L^\infty(B_{\l/4}^+)} \leq C\l^{-1}\|\bar{w}\|_{L^\infty(B_\l^+)} +C \|\bar{g}\|_{L^\infty(B_\l^+)}
\qquad \forall i \in 1, \hdots , n.
\]
In other words,
\[
|\pa_y(\pa_i u(X))| \leq Cy^{-a},
\]
where we see from Proposition~\ref{prop: derivative estimates} that $C$ depends on $a, n, \l, \| u\|_{L^\infty(B^+_{\l})}, \|f\|_{L^\infty(B^+_\l)}$, and $\|g\|_{L^{\infty}(B^+_\l)}$.
Arguing as in the proof of Corollary~\ref{cor: upgrade f is C1} completes the proof.
\end{proof}

%~~~~~~~~~~~~~~~~~~~~~~~~~~~~~~~~~~~~~~~~~~~~~~~~~~~~~~~%
\section{A Higher Order Boundary Harnack Estimate: $\G \in C^{k+2,\a}$ for $k \geq 0$}
\label{sec: HOBH}

In this section, we prove a pointwise higher order boundary Harnack estimate when $\G$ is at least $C^{2,\a}$.
This estimate, Proposition~\ref{prop: HOBH k pos}, and its $C^{1,\a}$ counterpart, Proposition~\ref{prop: HOBH 1 pos}, will play key roles in proving higher regularity of the free boundary in \eqref{eqn: obstacle problem}, as sketched in the introduction.
We refer the reader to Section~\ref{sec: Proof of Main Thm} for the details of how exactly this is accomplished.

Let $U\in C(B_1)$ be even in $y$ and normalized so that $U(\e_n/2) = 1$.
Suppose further that $U \equiv 0 $ on $\P$ and $U>0$ on $B_1\setminus \P$, and assume $U$ satisfies
\begin{equation}
\label{eqn: LAU}
L_a U = |y|^a\bigg(\frac{U_a}{r} T(x,r) + G(X)\bigg) \quad \text{in } B_1 \setminus \mathcal{P}
\end{equation}
where $T(x,r)$ is a polynomial of degree $k+1$ and 
\begin{equation*}
\label{eqn: H Est}
\|T\| \leq 1 \qquad\text{and}\qquad |G(X)| \leq r^{s-1}|X|^{k+1+\a}.
\end{equation*}
In Proposition~\ref{prop: pointwise Schauder Cka}, we show that if $\G \in C^{k+2,\a}$ with $\| \G\|_{C^{k+2,\a}}\leq 1$ and $U$ is as above, then $U$ takes the form
\begin{equation}
\label{eqn: U from schauder}
U = U_a (P_0 + O(|X|^{k+1+\a}))
\end{equation}
for some polynomial $P_0(x,r)$ of degree $k+1$ with $\|P_0\| \leq C$ and $U_a$ as defined in \eqref{eqn: Ua def}.
Formally, if we differentiate \eqref{eqn: U from schauder}, we find that 
%In addition, assume that the following derivative estimates hold:
\begin{equation}
\label{eqn: grad x}
\na_x U  = \frac{U_a}{r}\Big( s P_0 \nu + r \na_x P_0  + (\pa_r P_0 ) d \nu + O(|X|^{k+1+\a})\Big)
\end{equation}
and
\begin{equation}
\label{eqn: grad r}
\nab U \cdot \nab r = \frac{U_a}{r} \Big( sP_0 + (\pa_r P_0) r + \na_x P_0 \cdot (d\nu) + O(|X|^{k+1+\a})\Big).
\end{equation}
Rigorously justifying these expansions in our application to the fractional obstacle problem will require a delicate argument, which we present in Proposition~\ref{prop: deriv estimates}.
That said, in Proposition~\ref{prop: HOBH k pos}, we simply make the assumption that \eqref{eqn: grad x} and \eqref{eqn: grad r} hold.

\begin{remark}
\label{rmk: ahar case}
When $T \equiv G \equiv 0$, \eqref{eqn: grad x} and \eqref{eqn: grad r} follow by arguing as in Section~5 of \cite{DS3} and the Appendix of \cite{DS5}, using the regularity results in Section~\ref{sec: Prelim}.
\end{remark}

\begin{proposition}
\label{prop: HOBH k pos}
Let $\G \in C^{k+2, \a}$ with $\|\G\|_{C^{k+2,\a}} \leq 1$. 
Let $U, T, G$, and $P_0$ be as in \eqref{eqn: LAU}, \eqref{eqn: grad x}, and  \eqref{eqn: grad r}.
Suppose that $u \in C(B_1)$ is even in $y$ with $\|u\|_{L^\infty(B_1)} \leq 1$, vanishes on $\mathcal{P}$, and satisfies
\begin{equation*}
\La u =|y|^a\bigg(\frac{U_a}{r} R(x,r) + F(X)\bigg)   \quad \text{in }B_1 \setminus \mathcal{P}
\end{equation*}
where $R(x,r)$ is a polynomial of degree $k+1$ with $\| R\| \leq 1$ and 
\[
|F (X) | \leq r^{s-1}|X|^{k+1+\a}.
\]
Then, there exists a polynomial $P(x,r)$ of degree $k+2$ with $\| P \| \leq C$ such that 
\begin{equation*}
\bigg| \frac{u}{U} - P \bigg| \leq C|X|^{k+2+ \a}
\end{equation*}
for some constant $C = C(a,n,k,\a) > 0$.
\end{proposition}

Proposition~\ref{prop: HOBH k pos} is proved via a perturbation argument that relies on following result,
where we consider the specific case that $T, G, R, F \equiv 0$ and $\G$ is flat. 

\begin{proposition}
\label{prop: Schauder flat boundary} 
Suppose $u \in C(B_1)$ is even in $y$ with $\|u\|_{L^{\infty}(B_1)} \leq 1$ and satisfies
\begin{equation}
\label{eqn: G flat eq}
\begin{cases}
\La u = 0 &\text{in } B_1 \setminus \{ x_n \leq 0, \, y = 0\} \\
u = 0 &\text{on } \{ x_n \leq 0, \, y = 0\}.
\end{cases}
\end{equation}
Then, for any $k \geq 0$, there exists a polynomial $\bar{P}(x,r)$ of degree $k$ with $\|\bar{P}\| \leq C$ such that $\bar{U}_a \bar{P}$ is $a$-harmonic in $B_1 \setminus \{ x_n \leq 0,\, y = 0\}$ and 
\begin{equation}
\label{eqn: flat implies smooth}
|u -  \bar{U}_a \bar{P}| \leq C\bar U_a|X|^{k+1}
\end{equation}
for some constant $C = C(a,n,k) > 0$.
\end{proposition}

Recall that $\bar U_a$, defined in \eqref{eqn: Uabar def}, is $U_a$ when $\G$ is flat. 
We postpone the proof of Proposition~\ref{prop: Schauder flat boundary} until Section~\ref{sec: Flat}.
In order to proceed with proof of Proposition~\ref{prop: HOBH k pos}, we need to adapt the notion of approximating polynomials for $u/U$, introduced in \cite{DS5}, to our setting. 
Observe that
\begin{equation}
\begin{split}
\label{eqn: La U P}
L_a (U x^\mu r^m ) & = x^\mu r^m L_a U+ U \La (x^\mu r^m) + 2 |y|^a \na(x^\mu r^m ) \cdot \na U \\
&=|y|^a ({\rm I} +{\rm II} + {\rm III})
\end{split}
\end{equation}
where, letting $\bar{\imath}$ denote the $n$-tuple with a one in the $i$th position and zeros everywhere else,
\begin{align*} 
{\rm I} & =  \frac{U_a}{r} x^\mu r^m T(x,r) + x^\mu r^m G(X), \\
{\rm II} &=  \frac{U}{r} ( m (a+ m -d \kappa)x^\mu r^{m-1}  +\mu_i (\mu_i -1) r^{m+1 } x^{\mu-2\bar\imath} + 2dmr^{m-1} \na(x^\mu) \cdot \nu ),\\
{\rm III}& = 2(r^m \na_x U \cdot \na(x^m) + m x^\mu r^{m-1} \nab U \cdot \nab r).
\end{align*}
Up to a dilation, we can assume that
\begin{equation}
\label{eqn: smallness assumption} 
\|\G\|_{C^{k+2,\a} } \leq \vep, \qquad
\|T\|, \, \|R\| \leq \vep, \qquad\text{and}\qquad |G(X)|,\, |F(X)| \leq \vep r^{s-1}|X|^{k+1+\a}
\end{equation}
for any $\vep > 0$.
For $\vep > 0$ sufficiently and universally small, the constant term of $P_0$ in \eqref{eqn: U from schauder} is nonzero (see Remark~\ref{rmk: nonzero constant} below). So, up to multiplication by a constant, \eqref{eqn: U from schauder} takes the form 
\begin{equation}
\label{eqn: U and Ua}
U = U_a (1 + \vep Q_0 + \vep O(|X|^{k+1+\a})),
\end{equation}
where $Q_0(x,r)$ is a degree $k+1$ polynomial with zero constant term and $\|Q_0\| \leq 1$.

Taylor expansions of $\nu$, $\k$, and $d$ yield
\[
\nu_i = \delta_{in} + \cdots + \vep O(|X|^{k+1+\a}), \quad \k = \k(0) + \cdots + \vep O(|X|^{k+\a}), \quad\text{and}\quad d= x_n + \cdots + \vep O(|X|^{k+2+\a}).
\]
Hence, using \eqref{eqn: grad x} and \eqref{eqn: grad r} to  expand III, we find that ${\rm I, II}$, and ${\rm III}$ become
\begin{equation}
\label{eqn: I,II,III}
\begin{split}
{\rm I} &  = \frac{U_a}{r} s_{\sigma l}^{\mu m} x^\sigma r^l  + \vep O(r^{s-1}|X|^{k+1+\a}),\\
{\rm II} & =  \frac{ U}{r} \Big( m(a+ m +2 \mu_n ) x^\mu r^{m-1}  +\mu_i (\mu_i -1)  r^{m+1 } x^{\mu-2\bar\imath} + a_{\sigma l}^{\mu m} x^\sigma r^l + \vep O(|X|^{k+1+\a})\Big),\\
{\rm III} & = \frac{U_a}{r} \Big(2sr^m \mu_n x^{\mu-\bar n} + 2sm x^\mu r^{m-1} + b_{\sigma l}^{\mu m} x^\sigma r^l  + \vep O(|X|^{k+1+\a})\Big).
\end{split}
\end{equation}
Here, $s_{ \sigma l}^{\mu m}, a_{\sigma l}^{\mu m}$, and $b_{\sigma l}^{\mu m}$ are coefficients of monomials of degree at least $|\mu| + m$ and at most $k+1$; that is,
\[
s_{\sigma l}^{\mu m}, a_{\sigma l}^{\mu m}, b_{\sigma l}^{\mu m} \neq 0 \qquad\text{only if}\qquad |\mu|+m \leq |\sigma | + l \leq k+1.
\]
Furthermore, since the monomials $a_{\sigma l}^{\mu m} x^\sigma r^l $ and $b_{\sigma l}^{\mu m} x^\sigma r^l $ come from the Taylor expansions of $\nu, \k,$ and $d$ (which vanish when $\G$ is flat), recalling \eqref{eqn: smallness assumption}, we have that
\[
|s_{\sigma l}^{\mu m}|,
|a_{\sigma l}^{\mu m}|, |b_{\sigma l}^{\mu m}| \leq C\vep.
\]
Therefore, from \eqref{eqn: La U P}, \eqref{eqn: U and Ua}, and \eqref{eqn: I,II,III}, we determine that
\begin{equation*}
\begin{split}
L_a (U x^\mu r^m)= 
|y|^a \bigg(\frac{U_a}{r}& \Big( 
m x^\mu r^{m-1} (1 + m +2 \mu_n )  + 2s\mu_n r^m x^{\mu-\bar n} +\mu_i (\mu_i -1)  r^{m+1 } x^{\mu-2\bar \imath}   + c_{\sigma l}^{\mu m} x^\sigma r^l \Big) \\
&+ \vep O(r^{s-1}|X|^{k+1+\a})\bigg)
\end{split}
\end{equation*}
where 
\[
c_{\sigma l}^{\mu m} \neq 0 \qquad \text{only if}\qquad |\mu|+m \leq |\sigma | + l \leq k+1 \qquad\text{and}\qquad |c_{\sigma l}^{\mu m}| \leq C\vep.
\]
Thus, given a polynomial $P(x,r) = p_{\mu m} x^\mu r^m$ of degree $k+2$, 
\[ 
L_a(UP) =  |y|^a \bigg(\frac{U_a}{r}A_{\sigma l} x^\sigma r^l + h(X)\bigg) 
\]
where $|\sigma | + l \leq k+1$,
\begin{equation}
\label{eqn: approx poly error term HOBH}
|h(X)| \leq C \vep\|P\| r^{s-1}|X|^{k+1+\a},
\end{equation}
and
\begin{equation}
\label{eqn: poly coeffs HOBH}
A_{\sigma l}  = (l +1)(l+2+2\sigma_n) p_{\sigma, l+1} + 2s(\sigma_n + 1) p_{\sigma+\bar n, l} + (\sigma_i +1)(\sigma_i + 2) p_{\sigma + 2\bar \imath, l-1} + c_{\sigma l}^{\mu m} p_{\mu m}.
\end{equation}
From \eqref{eqn: poly coeffs HOBH}, we see that $p_{\sigma,l+1}$ can be expressed in terms of $A_{\sigma l}$, a linear combination of $p_{\mu m}$ for $\mu + m \leq |\sigma| + l$, and a linear combination of $p_{\mu m}$ for $\mu + m \leq |\sigma| + l$ and $m \leq l$.
Consequently, the coefficients $p_{\mu m}$ are uniquely determined by the linear system \eqref{eqn: poly coeffs HOBH} given $A_{\sigma l}$ and $p_{\mu 0}$.

\begin{definition}\label{def: approx poly}
Let $u$ and $U$ be as in Proposition~\ref{prop: HOBH k pos}.
A polynomial $P(x,r)$ of degree $k+2$ is {\it approximating} for $u/U$ if the coefficients $A_{\sigma l}$ coincide with the coefficients of $R$.
\end{definition}
 
Before we prove Proposition~\ref{prop: HOBH k pos}, let us make two remarks and state a lemma.

\begin{remark}
\label{rem: bHi and Ua}
While $U_a$ is not $a$-harmonic in $B_1 \setminus \mathcal{P}$, it is comparable in $B_1$ to a function $V_a$ that is $a$-harmonic in $B_1 \setminus \mathcal{P}$. Indeed, using the upper and lower barriers
$
V_{\pm} := (1 \pm r/2)U_a,
$
one can construct such a function $V_a$ by Perron's method. \end{remark}

\begin{remark}
\label{rmk: nonzero constant}
Up to an initial dilation, taking $\|\G\|_{C^{k+2,\a}} \leq \vep$ for a universally small $\vep > 0$, the constant term of $P_0$ in \eqref{eqn: U from schauder} is nonzero.
If $U$ and $U_a$ were $a$-harmonic in $B_1 \setminus \mathcal{P}$, this would follow directly from the boundary Harnack estimate applied to $U_a/U$ without a dilation.
By Remark~\ref{rem: bHi and Ua}, $U_a$ is comparable to the $a$-harmonic function $V_a$.
For $\vep > 0$ sufficiently small (universally so), we will find that $U$ is also comparable to an $a$-harmonic function $W$; hence, we can effectively apply the boundary Harnack estimate to $U_a/U$ passing through the quotient $V_a/W$ to conclude.
More specifically, after dilating, let us normalize so that $U(\e_n/2) = 1$.
First, let $W$ satisfy
\[
\begin{cases}
L_a W = 0 & \text{in } B_1 \setminus \P\\
W = U & \text{on } \pa B_1 \cup \P.
\end{cases}
\]
The strong maximum principle ensures that $W$ is positive in $B_1 \setminus \mathcal{P}$. Applying the boundary Harnack estimate to $W$ and $V_a$, Remark~\ref{rem: bHi and Ua} implies that
\[
cU_a \leq \frac{W}{W(\e_n/2)} \leq CU_a.
\]
Second, let $V := U - W$, and observe that
\[
\begin{cases}
|L_a V| \leq C\vep  |y|^a r^{s-1} & \text{in }B_1 \setminus \P\\
V = 0 & \text{on } \pa B_1 \cup \P.
\end{cases}
\]
Lemma~\ref{lem: Simple Barriers} then shows that 
\[
|V| \leq  C\vep U_a.
\]
So, $1 - C\vep \leq W(\e_n/2) \leq 1 + C\vep$ and for $\vep > 0$ small, 
\[
cU_a \leq U \leq CU_a.
\]
Now, if the constant term of $P_0$ were zero, then \eqref{eqn: U from schauder} would yield $cU_a \leq U \leq CU_a|X|,
$
which is impossible.
\end{remark}

In addition to its use in Remark~\ref{rmk: nonzero constant} above, the following lemma will be used at several other points. 
The proof follows by considering the upper and lower barriers $v_{\pm} := \pm \, C (U_a - U_a^{1/s})$.

\begin{lemma}
\label{lem: Simple Barriers}
Let $v\in C(B_1)$ satisfy
\[
\begin{cases}
|\La v| \leq |y|^a r^{s-1} &\text{in } B_1 \setminus \mathcal{P} \\
v = 0 &\text{on } \partial B_1 \cup \mathcal{P}.
\end{cases}
\]
Then,
\[
|v| \leq CU_a
\]
for some constant $C = C(a) > 0$.
\end{lemma}

We are now in a position to prove Proposition~\ref{prop: HOBH k pos}.

\begin{proof}[Proof of Proposition~\ref{prop: HOBH k pos}]
First, let $\vep > 0$ in \eqref{eqn: smallness assumption} be sufficiently small so that Remark~\ref{rmk: nonzero constant} holds.
Then, solving a system of linear equations as discussed above, we obtain an initial approximating polynomial $Q^0(x,r)$ of degree $k+2$ for $u/U$.
Up to multiplying $u$ by a small constant and further decreasing $\vep > 0$ (recall that $\|Q^0\| \leq C\vep$), we can assume that
\[
\| Q^0\| \leq 1, \qquad 
\|u-UQ^0\|_{L^{\infty}(B_1)}\leq 1, \qquad\text{and}\qquad |\La [u-UQ^0](X)| \leq |y|^a \vep r^{s-1}|X|^{k+1+\a}.  
\]

\noindent{\it Step 1: There exists $0<\rho<1$, depending on $a,n,k$, and $\alpha$, such that, up to further decreasing $\vep>0$, the following holds.  If there exists a polynomial $Q$ of degree $k+2$ that is approximating for $u/U$ with $\| Q \|\leq 1$ and 
\[
\| u - UQ \|_{L^\infty(B_\l)} \leq \l^{k+2+\a+s},
\]
then there exists a polynomial $Q'$ of degree $k+2$ that is approximating for $u/U$ with
\begin{equation*}
\label{eqn: HOBH at scale rho lambda}
\| u - UQ' \|_{L^\infty(B_{\rho \l})} \leq (\rho\l)^{k+2+\a+s}
\end{equation*}
and
\[ 
\| Q' - Q\|_{L^\infty(B_\l)} \leq C\l^{k+2+\a}.
\]
for some constant $C = C(a,n,k,\a) >0$.}\\

For any $\l > 0$, define the rescalings
\begin{equation}
\label{eqn: rescaling}
\mathcal{P}_\l := \frac{1}{\l}\mathcal{P}, \qquad r_\l(X) := \frac{r(\l X)}{\l}, \qquad U_{a,\l}(X) := \frac{U_a(\l X)}{\l^s}, \qquad U_{\l}(X) := \frac{U(\l X)}{\l^s},
\end{equation}
and
\[
\tilde{u}(X) : = \frac{[u-UQ](\l X)}{\l^{k+2+\a+s}}.
\]
Thus, $\|\tilde{u}\|_{L^{\infty}(B_1)} \leq 1$, and by \eqref{eqn: approx poly error term HOBH}, we have that
\begin{equation}
\label{eqn: u tilde RHS}
|\La \tilde{u}(X)| \leq C \vep |y|^ar_{\l}^{s-1}|X|^{k+1+\alpha}.
\end{equation}
Let $w$ be the unique solution to
\begin{equation}
\label{eqn: a-har replacement}
\begin{cases}
\La w =0 &\text{in } B_1 \setminus \mathcal{P_\l } \\
w =0 &\text{on } \mathcal{P_\l} \\
w  = \tilde{u} &\text{on } \pa B_1.
\end{cases}
\end{equation}
Notice that $w$ is even in $y$ by the symmetry of the domain and boundary data and $\|w\|_{L^\infty(B_1)} \leq 1$ by the maximum principle.
Since $\mathcal{P_\l}$ has uniformly positive $\La$-capacity independently of $\vep$ and $\l$, $w$ is uniformly H\"older continuous in compact subsets of $B_1$.
So, letting $\bar w$ be the solution of \eqref{eqn: a-har replacement} in $B_1 \setminus \{ x_n \leq 0,\, y = 0 \}$,
by compactness, $w$ is uniformly close to $\bar w$ if $\vep$ is sufficiently small (universally so). Indeed, recall that $\G \to \{ x_n = 0,\, y = 0 \}$ in $C^{k+2,\alpha}$ as $\vep \to 0$.
Furthermore, thanks to \eqref{eqn: U and Ua}, we have that $U_\l \to \bar U_a$ uniformly as $\vep \to 0$.

Proposition~\ref{prop: Schauder flat boundary} ensures that
there exists a polynomial
\[
\bar{P}(X) := \bar{p}_{\mu m}x^\mu |z|^m 
\]
of degree $k+2$ such that $ \|\bar{P}\| \leq C$, $\bar{U}_a \bar{P}$ is $a$-harmonic in the set $B_1 \setminus \{x_n \leq 0,\, y = 0 \}$, and 
\[
\|\bar{w} - \bar{U}_a\bar{P}\|_{L^\infty(B_\rho)}  \leq  C\rho^{k+3+s}.
\]
Notice that the $a$-harmonicity of $\bar{U}_a \bar{P}$ implies that
\begin{equation}
\label{eqn: a har poly coeff}
(l+1)(l+2+2\sigma_n)\bar{p}_{\sigma,l+1} + 2s(\sigma_n + 1)\bar{p}_{\sigma+\bar{n},l} + (\sigma_i + 1)(\sigma_i+2)\bar{p}_{\sigma + 2\bar{\imath},l-1} = 0 \qquad \forall (\sigma, l).
\end{equation}
Therefore, choosing $\rho$ and then $\vep$ sufficiently small depending on $a,n,k$, and $\alpha$, we find that
\begin{equation}
\label{eqn: approx with flat HOBH}
\begin{split}
\|w - U_\l \tilde{P}\|_{L^\infty(B_\rho)}& 
\leq \|w - \bar{w}\|_{L^\infty(B_{1/2})} + \|U_\l \tilde{P} - \bar{U}_a\bar{P}\|_{L^\infty(B_{1/2})} + \|\bar{w} - \bar{U}_a\bar{P}\|_{L^\infty(B_\rho)} \\
&\leq \frac{1}{4} \rho^{k+2 + \a +s}
\end{split}
\end{equation}
where $\tilde{P}(X) := \bar{p}_{\mu m}x^\mu r_\l^m$ has the same coefficients as $\bar P$.
Now, set $v := \tilde{u} - w$.
From \eqref{eqn: u tilde RHS}, we find that
\[
\begin{cases}
|\La v| \leq \vep|y|^a r_\l^{s-1} &\text{in } B_1 \setminus \mathcal{P_\l} \\
v = 0 &\text{on } \partial B_1 \cup \mathcal{P_\l}.
\end{cases}
\]
From Lemma~\ref{lem: Simple Barriers} and \eqref{eqn: U and Ua}, we deduce that 
\begin{equation}
\label{eqn: barriers HOBH}
|v| \leq C\vep U_{a,\l} \leq C \vep U_\l.
\end{equation}

Then, combining \eqref{eqn: barriers HOBH} and \eqref{eqn: approx with flat HOBH} and further decreasing $\vep$ depending on $\rho, a,n,k$, and $\alpha$, we find that  
\[ 
\| \tilde{u} - U_\l\tilde{P} \|_{L^{\infty}(B_{\rho})} \leq \|v\|_{L^{\infty}(B_{\rho})} + \| w - U_\l \tilde{P} \|_{L^{\infty}(B_{\rho})} \leq \frac{1}{2}\rho^{k+2+\a +s}.
% \leq  C \vep + C\rho^{k+3+s}.
\]
Rescaling implies that
\[
\| u - U \tilde{Q}\|_{L^\infty(B_{\rho \l)}} \leq \frac{1}{2}(\rho \l)^{k+2+\a+s}
\]
with $\tilde{Q}(X):= Q(X) + \l^{k+2+\a}\tilde{P}(X/\l)$.

To conclude, we must alter $\tilde{Q}$ to make it an approximating polynomial for $u/U$ by replacing $\tilde{P}(X/\l)$ with another polynomial $P'(X/\l)$. As $Q$ is already an approximating polynomial for $u/U$, we need the coefficients $p'_{\mu m}$ of $P'$ to satisfy the system
\begin{equation}
\label{eqn: correction poly}
(l+1)(l+2+2\sigma_n) p'_{\sigma, l+1} + 2s(\sigma_n + 1) p'_{\sigma+\bar n, l} + (\sigma_i +1)(\sigma_i + 2) p'_{\sigma + 2\bar \imath, l-1} +  \tilde{c}_{\sigma l}^{\mu m} p'_{\mu m} = 0 \qquad \forall (\sigma, l)
\end{equation}
where
\[
\tilde{c}_{\sigma l}^{\mu m} : = \l^{|\sig| + l + 1 - |\mu| - m} c_{\sigma l}^{\mu m}.
\]
Notice that $|\tilde{c}_{\sigma l}^{\mu m}| \leq |c_{\sigma l}^{\mu m}| \leq C\vep$.
Furthermore, subtracting \eqref{eqn: correction poly} from \eqref{eqn: a har poly coeff}, we see that $P' - \tilde{P}$ solves the system \eqref{eqn: correction poly} with right-hand side
\[
A_{\sig l} = \tilde{c}_{\sigma l}^{\mu m}\bar{p}_{\mu m}.
\]
Hence, $|A_{\sig l}| \leq C\vep$, and choosing $p'_{\mu 0} = \bar{p}_{\mu 0}$, we uniquely determine $P'$ and find that
\[
\| P' - \tilde{P} \| \leq C\vep.
\]
Setting $Q'(X) := Q(X) + \l^{k+2+\a}P'(X/\l)$ concludes Step 1. \\

\noindent {\it Step 2: Iteration and Upgrade.} \\

Iterating Step 1, letting $\l = \rho^j$ for $j = 0, 1, 2, \dots $, we find a limiting approximating polynomial $P$ such that $\|P\| \leq C$ and
\[
\| u - UP \|_{L^{\infty}(B_{\rho^j})} \leq C\rho^{j(k+2+\a+s)} \qquad \forall j \in \N.
\]
To upgrade this inequality to 
\begin{equation}\label{eqn: upgraded}
|u - UP| \leq CU|X|^{k+2+\a},
\end{equation}
we argue as in Step 1 in $B_1 \setminus \mathcal{P}_\l$.
Setting
\[
\tilde{u}(X) := \frac{[u - UP](\l X)}{\l^{k+2 +\a +s}},
\]
we have that 
\[
|\tilde{u}| \leq |w| + |v| \leq C U_{a,\l} \leq CU_\l \quad\text{in } B_{1/2}.
\]
Indeed, that $v$ and $w$ are controlled by $U_{a,\l}$ comes from Lemma~\ref{lem: Simple Barriers} and an application of the boundary Harnack estimate (cf. Remark~\ref{rem: bHi and Ua}), while the last inequality comes from \eqref{eqn: U and Ua}.
Thus, after rescaling, we deduce that \eqref{eqn: upgraded} holds since $0 < \l \leq 1$ was arbitrary.
\end{proof}

Keeping Remark~\ref{rmk: ahar case} in mind, if $U$ is $a$-harmonic, then \eqref{eqn: grad x} and \eqref{eqn: grad r} hold.
So, a consequence of Proposition~\ref{prop: HOBH k pos} and Proposition~\ref{prop: HOBH 1 pos}, its $C^{1,\a}$ analogue, is the following full generalization of \cite[Theorem~2.3]{DS5}.

\begin{theorem} 
\label{thm: General DS5}
Suppose $ 0 \in \G := \partial_{\R^n} \P $ with $\G \in C^{k,\a}$ for $k \geq 1$ and $\| \G \|_{C^{k,\a}}\leq 1$. 
If $u$ and $U$ are even in $x_{n+1}$, $\|u\|_{L^\infty(B_1)} \leq 1$,
\[
\begin{cases}
\La u = |y|^a\frac{U_a}{r}f &\text{in } B_1 \setminus \mathcal{P}\\
u = 0 &\text{on } \mathcal{P}
\end{cases}
\]
for 
\[
f \in C^{k-1,\a}_{x,r}(\G \cap B_1) \qquad\text{with}\qquad \|f\|_{C^{k-1,\a}_{x,r}(\G \cap B_1)} \leq 1,
\]
and $U > 0$ is $a$-harmonic in $B_1 \setminus \P$ with $U(\nu(0)/2) = 1$, where $\nu$ is the outer unit normal to $\P$,
then
\[
\bigg\|\frac{u}{U}\bigg\|_{C^{k,a}_{x,r}(\G \cap B_{1/2})} \leq C
\]
for some $C = C(a,n,k,\a) > 0$.
\end{theorem}

%~~~~~~~~~~~~~~~~~~~~~~~~~~~~~~~~~~~~~~~~~~~~~~~~~~~~~~~%
\section{Schauder Estimates: $\G \in C^{k+2,\alpha}$ for $k \geq 0$}
\label{sec: Schauder Cka}
 
In Proposition~\ref{prop: HOBH k pos}, we were crucially able to approximate $U$ in terms of $U_a P_0$, where $P_0 = P_0(x,r)$ is a polynomial of degree $k+1$. 
This approximation, \eqref{eqn: U from schauder}, is a consequence of the Schauder estimates of Proposition~\ref{prop: pointwise Schauder Cka} below. 
These Schauder estimates roughly say if $u$ satisfies
\[
\La u = |y|^a \frac{U_a}{r} f \quad\text{in } B_1 \setminus \P \qquad\text{and}\qquad u=0 \quad\text{on } \P,
\] then $u$ gains regularity in terms of the regularity of $f$ and $\G$.
More precisely, we find that $u/U_a \in C^{k+1,\a}_{x,r}(0)$ if $f \in C^{k,\alpha}_{x,r}(0)$ and $\G \in C^{k+2,\a}$.

\begin{proposition}
\label{prop: pointwise Schauder Cka}
Let $\G \in C^{k+2,\alpha}$ with $\| \G \|_{C^{k+2,\a}} \leq 1$.
Suppose $u\in C(B_1)$ is even in $y$ with $\|u\|_{L^{\infty}(B_1)}\leq 1$, vanishes on $\mathcal{P}$, and satisfies
\begin{equation*}
\La u(X) = |y|^a\bigg(\frac{U_a}{r}R(x,r)+ F(X)\bigg) \quad\text{in } B_1 \setminus \mathcal{P}
\end{equation*}
where $R(x,r)$ is a polynomial of degree $k$ with $\|R\| \leq 1$ and 
\begin{equation*}
|F(X)| \leq r^{s-1}|X|^{k + \alpha}.
\end{equation*}
Then, there exists a polynomial $P_0(x,r)$ of degree $k+1$ with $\|P_0\| \leq C$ such that 
\begin{equation*}
| u - U_a P_0 | \leq C U_a |X|^{k+1+\a}
\end{equation*}
and
\begin{equation*}
|\La (u - U_a P_0)| \leq C|y|^a r^{s-1}|X|^{k+\alpha}\quad\text{in } B_1 \setminus \mathcal{P}
\end{equation*}
for some constant $C = C(a,n,k,\a) >0$.
\end{proposition}

To prove Proposition~\ref{prop: pointwise Schauder Cka}, we must first extend the appropriate notion of approximating polynomial to this setting.
We compute that
\begin{equation*}
\begin{split}
\La (U_a x^\mu r^m) &= |y|^a \frac{U_a}{r} \Big(-( dm+sr)\kappa x^\mu r^{m-1} + m(m+1)x^\mu r^{m-1} \\
&\hspace{2.0cm}+ 2r^{m-1}(dm + sr) \nu \cdot \nabla_x x^\mu + \mu_i(\mu_i - 1)x^{\mu-2\bar{\imath}}r^{m+1} \Big).
\end{split}
\end{equation*}
Each of the functions $\nu_i, \kappa$, and $d$ can be written as the sum of a degree $k$ polynomial in $x$ and a $C^{k,\alpha}$ function in $x$ whose derivatives vanish up to order $k$. 
The lowest degree terms in the Taylor expansion at zero of $\nu_i, \kappa$, and $d$ are $\delta_{i n}, \kappa(0)$, and $x_n$ respectively. 
Hence, grouping terms by degree up to order $k$ and the remainder, we see that
\begin{equation*}
\begin{split}
\La (U_a x^\mu r^m) &= |y|^a \frac{U_a}{r} \Big(m(m+1+2\mu_n)x^\mu r^{m-1} +2s\mu_nx^{\mu-\bar{n}}r^{m}\\
&\hspace{2.0cm}+ \mu_i(\mu_i - 1)x^{\mu-2\bar{\imath}}r^{m+1} + c_{\sigma l}^{\mu m}x^\sigma r^l + h^{\mu m}(x,r) \Big).
\end{split}
\end{equation*}
Here,
$c_{\sigma l}^{\mu m} \neq 0$ only if $|\mu|+m \leq |\sigma|+l \leq k$.
Also,
\[
h^{\mu m}(x,r) := r^m h^{\mu}_{m}(x) + m r^{m-1}h^{\mu}_{m-1}(x),
\]
and $h^{\mu}_{m}, h^{\mu}_{m-1} \in C^{k,\alpha}(B_1^*)$ have vanishing derivatives up to order $k-m$ and $k - (m-1)$ at zero respectively.
The coefficients $c_{\sigma l}^{\mu m}$ are all linear combinations of the Taylor coefficients at the origin of $\kappa, d \kappa, \nu_i$, and $ d\nu_i$, which vanish if $\G$ is flat. 
After a dilation making $\|\G\|_{C^{k+2,\alpha}} \leq \vep$, we may assume that
\[
|c_{\sigma l}^{\mu m}| \leq \vep, \qquad \|h^{\mu}_{m}\|_{C^{k,\alpha}(B_1^*)} \leq \vep, \qquad\text{and}\qquad \|h^{\mu}_{m-1}\|_{C^{k,\alpha}(B_1^*)} \leq \vep.
\]
Therefore, if $P(x,r) = p_{\mu m}x^\mu r^m$ is a polynomial of degree $k+1$, then
\[
\La (U_aP) = |y|^a\frac{U_a}{r}\Big(A_{\sigma l}x^{\sigma}r^l + h(x,r)\Big)
\]
where $|\sigma| + l \leq k$,
\begin{equation}
\label{eqn: approx poly coeff}
A_{\sigma l} = (l+1)(l+2+2\sigma_n)p_{\sigma,l+1} + 2s(\sigma_n + 1)p_{\sigma+\bar{n},l} + (\sigma_i + 1)(\sigma_i+2)p_{\sigma + 2\bar{\imath},l-1} + c_{\sigma l}^{\mu m}p_{\mu m},
\end{equation}
and
\[
h(x,r) := \sum_{m=0}^k r^m h_m(x)
\]
for $h_m \in C^{k,\alpha}(B_1^*)$ with vanishing derivatives up to order $k-m$ at zero.
Assuming that $\|\G\|_{C^{k+2,\alpha}} \leq \vep$, we have
\begin{equation*}
|h(X)| \leq \vep \|P\||X|^{k+\alpha}.
\end{equation*}
Considering \eqref{eqn: approx poly coeff}, we see that $p_{\sigma,l+1}$ can be expressed in terms of $A_{\sigma l}$, a linear combination of $p_{\mu m}$ for $\mu + m \leq |\sigma| + l $, and a linear combination of $p_{\mu m}$ for $\mu + m \leq |\sigma| + l $ and $m \leq l$.
Thus, the coefficients $p_{\mu m}$ are uniquely determined by the linear system \eqref{eqn: approx poly coeff} given $A_{\sigma l}$ and $p_{\mu 0}$.

\begin{definition}
Let $u$ be as in Proposition~\ref{prop: pointwise Schauder Cka}.
We call a polynomial $P(x,r)$ of degree $k+1$ {\it approximating} for $u/U_a$ if the coefficients $A_{\sigma l}$ coincide with the coefficients of $R$.
\end{definition}

\begin{remark}
\label{rem: U v. Ua expansions}
Observe that 
\[
L_a U_a =  -|y|^a \frac{U_a}{r} s \k(x)
\qquad\text{and}\qquad
L_a U = |y|^a \bigg(\frac{U_a}{r}T(x,r) + G(X)\bigg).
\]
Since $\G \in C^{k+2,\a}$, the mean curvature $\k = \k(x)$ does not possess enough regularity to yield the same order error as $G$ after being expanded.
Indeed, letting $g := \k - T$ where $T$ is the $k$th order Taylor polynomial of $\k$ at the origin, we see that
\[
\frac{U_a}{r}|g(X)| \leq r^{s-1}|X|^{k+\a} \qquad\text{while}\qquad |G(X)| \leq r^{s-1}|X|^{k+1+\a}.
\]
This discrepancy lies at the heart of the difference in approximating $u/U_a$ and $u/U$.
\end{remark}

With the correct notion of approximating polynomial in hand, the proof of Proposition~\ref{prop: pointwise Schauder Cka} is now identical to that of Proposition~\ref{prop: HOBH k pos} upon replacing $U$ with $U_a$; it is therefore omitted.

%~~~~~~~~~~~~~~~~~~~~~~~~~~~~~~~~~~~~~~~~~~~~~~~~~~~~~~~%
\section{The Low Regularity Case: $\G \in C^{1,\a}$}
\label{sec: C1a}
 
The goal of this section is to prove Proposition~\ref{prop: HOBH 1 pos}, which extends the higher order boundary Harnack estimate of Proposition~\ref{prop: HOBH k pos} to the case when $\G$ is only of class $C^{1,\a}$.
In this case, the functions $r$ and $U_a$ introduced in Section~\ref{subsec: Notation} do not possess enough regularity to directly extend the proof of Proposition~\ref{prop: pointwise Schauder Cka} or the notion of approximating polynomial for $u/U$ in Definition~\ref{def: approx poly}. 
Following \cite{DS5}, this is rectified by working with regularizations of $r$ and $U_a$, denoted by $r_*$ and $U_{a,*}$ respectively. 
The following lemma contains estimates which will allow us to replace $r$ and $U_a$ by their regularizations when needed. 
The construction and the proofs of these estimates can be found in the Appendix. 

\begin{lemma}
\label{lem: Approximations}
Let $\| \G\|_{C^{1,\a}} \leq 1$.
There exist smooth functions $r_*$ and $U_{a,*}$ such that 
\[
\bigg| \frac{r_*}{r} -1\bigg| \leq C_{*} r^\a, \qquad \bigg| \frac{U_{a,*}}{U_a} - 1\bigg|\leq C_{*} r^\a,
\]
\[
|\na r_* - \na r| \leq C_{*} r^\a, \qquad |\pa_y r_* - \pa_y r| \leq C_{*}|y|^a U_a r^{s -1 +\a}, \qquad \bigg|\frac{|\na U_{a,*}|}{|\na U_a |} - 1\bigg| \leq C_{*} r^\a,
\]
\[
\bigg|L_a r_* - \frac{2(1-s)|y|^a}{r} \bigg| \leq C_{*}|y|^a r^{\a -1}, \qquad\text{and}\qquad | \La U_{a,*}| \leq C_{*}|y|^a r^{s - 2 + \a}
\]
where $C_{*} = C_{*}(a,n,\a) > 0$.
If $\| \G\|_{C^{1,\a}} \leq \vep$, then each inequality holds with the right-hand side multiplied by $\vep$.
\end{lemma}

The following pointwise Schauder estimate plays the role of Proposition~\ref{prop: pointwise Schauder Cka} in the case when $\G$ is $C^{1,\a}.$
\begin{proposition}
\label{prop: C1a ptwise Schauder}
Let $\G \in C^{1,\a}$ with $\| \G\|_{C^{1,\a}} \leq 1$.
Suppose $u \in C(B_1)$ is even in $y$ with $\|u\|_{L^\infty(B_1)} \leq 1$, vanishes on $\mathcal{P}$, and satisfies
\begin{equation}
\label{eqn: La u C1a schauder}
|\La u| \leq |y|^a r^{s-2+\a} \quad \text{in } B_1 \setminus \mathcal{P}.
\end{equation}
Then, there exists a constant $p'$ with $|p'| \leq C$ such that 
\[ 
|u-p' U_a| \leq CU_a|X|^\a
\]
for some constant $C = C(a,n,\a) > 0$.
\end{proposition}

Note that even though Proposition~\ref{prop: C1a ptwise Schauder} is stated just at the origin, it holds uniformly at all points $\G \cap B_{1/2}$ since the assumption on the right-hand side in \eqref{eqn: La u C1a schauder} does not distinguish the origin.
%In order to prove Propositions~\ref{prop: HOBH 1 pos} and \ref{prop: C1a ptwise Schauder}, as in \cite{DS5}, we replace $r$ and $U_a$ with regularizations when the operator must be computed.
The proof of Proposition~\ref{prop: C1a ptwise Schauder} is quite similar to that of Propositions~\ref{prop: HOBH k pos} and \ref{prop: pointwise Schauder Cka}, but we include it to demonstrate how $U_{a,*}$ is used.
In the proof, we will make use of the following lemma, whose proof via a barrier argument is given in the Appendix. 

\begin{lemma}
\label{lem: C1a barrier} 
Assume $\| \G \|_{C^{1,\a}} \leq \vep$ with $\a \in (0, 1-s)$, and suppose $u$ satisfies 
\[
\begin{cases} |\La u| \leq |y|^a r^{\a -2 +s} &\text{in } B_1 \setminus \mathcal{P}\\
u = 0 &\text{on } \pa B_1 \cup \mathcal{P}.
\end{cases}
\]
If $\vep > 0$ is sufficiently small, depending on $a$, $n$, and $\a$, then
\[
|u| \leq C U_a
\]
for some $C = C(a,n,\a) > 0$.
\end{lemma}

\begin{proof}[Proof of Proposition~\ref{prop: C1a ptwise Schauder}]
Up to a dilation, we may assume that
\begin{equation*}
\|\G\|_{C^{1,\alpha}} \leq \vep \qquad\text{and}\qquad |\La u| \leq \vep|y|^ar^{s-2+\a}
\end{equation*}
for any $\vep > 0$, in particular, for $\vep$ small enough to apply Lemma~\ref{lem: C1a barrier}.\\

\noindent{\it Step 1: There exists $0<\rho<1$, depending on $a,n$, and $\alpha$, such that, up to further decreasing $\vep>0$, the following holds.  
If there exists a constant $q$ such that $|q| \leq 1$ and
\[
\|u-qU_a\|_{L^\infty(B_{\l})} \leq \l^{\a+s},
\]
then there exists a constant $q'$ with $|q'|\leq C$ such that
\begin{equation*}
\| u - Uq' \|_{L^\infty(B_{\rho \l})} \leq (\rho\l)^{\a + s}
\end{equation*}
and
\[ 
| q' -q|\leq C\l^{\a}
\]
for some constant $C = C(a,n,\a) > 0$.}\\

Define $\mathcal{P}_\l, r_\l$, and $U_{a,\l}$ as in \eqref{eqn: rescaling}, and consider the rescaling
\[ 
\tilde{u}(X) := \frac{[u - q U_{a,*}](\l X)}{2C_*\l^{\a +s}}.
\]
Note that $\|\tilde{u}\|_{L^\infty(B_1)} \leq 1$ by Lemma~\ref{lem: Approximations}.
Let $w$ be the unique solution of 
\begin{equation*}
\label{eqn: C1a problem}
\begin{cases}
\La w = 0 &\text{in } B_1\setminus \mathcal{P}\\
w = 0 &\text{on } \mathcal{P}\\
w = \tilde{u} &\text{on } \pa B_1.\\
\end{cases}
\end{equation*}
Observe that $w$ is even in $y$ and $\|w\|_{L^\infty(B_1)} \leq 1$.
By compactness, $w \to \bar{w}$ locally uniformly as $\vep \to 0$ where $\bar{w}$ vanishes on $\{x_n \leq 0,\, y=0\}$ and is such that $\La \bar{w} = 0$ in $B_1 \setminus \{x_n \leq 0,\, y=0\}$.
Proposition~\ref{prop: Schauder flat boundary} ensures the existence of a constant $\bar p$ with $|\bar p|\leq C$ such that, choosing $\rho$ and then $\vep $ sufficiently small, depending on $a, n$, and $\a$,
\[
\|w - \bar{p} U_{a,\l}\|_{L^\infty(B_\rho)} \leq \|w - \bar{p} \bar U_{a}\|_{L^\infty(B_\rho)}  + \|\bar{p}(\bar{U}_{a}- U_{a,\l})\|_{L^\infty(B_\rho)}\leq \frac{1}{8C_*}\rho^{\a+s}.
\]
Since $v := \tilde{u} - w$ satisfies
\[
\begin{cases}
|\La v| \leq \vep|y|^a r_\l^{s-2+\a} &\text{in } B_1\\
v = 0 &\text{on } \pa B_1 \cup \mathcal{P_\l},
\end{cases}
\]
Lemma~\ref{lem: C1a barrier} shows that $|v| \leq C\vep U_{a,\l}$.
Then, up to further decreasing $\vep$, we deduce that
\[
\|\tilde{u}-\bar{p}U_{a,\l}\|_{L^\infty(B_{\rho})} \leq \frac{1}{4C_*}\rho^{\a+s}.
\]
In terms of $u$, this implies that 
\[
\|u - qU_{a,*} - 2C_*\l^{a+s}\bar{p} U_{a}\|_{L^\infty(B_{\rho \l})} \leq \frac{1}{2}(\rho \l)^{\a +s}.
\]
Consequently, by Lemma~\ref{lem: Approximations}, further decreasing $\vep$ if necessary, we find that
\[
\|u - q'U_{a} \|_{L^\infty(B_{\rho \l})} \leq(\rho \l)^{\a +s} 
\qquad\text{and}\qquad
| q'-q |\leq C \l^{\a +s}
%\| q'-q \|_{L^\infty(B_\l)} \leq C \l^{\a +s}
\]
where $q' := q + 2C_*\l^{\a+s} \bar{p}$.
\\\\
\noindent{\it Step 2: Iteration and Upgrade.}\\

Iterating Step 1, letting $\l = \rho^j$, we find that there exists a limiting constant $p'$ such that 
\[
\|u - p' U_a\|_{L^\infty(B_{\rho^j})} \leq C\rho^{j(\a+s)} \qquad\forall j \in \N.
\]
To conclude, we must upgrade this inequality to 
\[
|u - p' U_a| \leq CU_a|X|^\a.
\] 
Arguing as in Step 1, in $B_1 \setminus \mathcal{P}_\l$, with
\[
\tilde{u}(X) := \frac{[u -p'U_{a,*}](\l X)}{\l^{\a +s}},
\]
we have that 
\[
|\tilde{u}| \leq |w| + |v| \leq C U_{a,\l} \quad\text{in } B_{1/2}.
\]
Indeed, the bound on $v$ comes from Lemma~\ref{lem: C1a barrier}, and the bound on $w$ comes from an application of the boundary Harnack estimate (see Remark~\ref{rem: bHi and Ua}) and Lemma~\ref{lem: Approximations}.\footnote{\,Here, we use the upper and lower barriers $U_{a,*} \mp U_{a,*}^{1+\a/s}$ in Perron's method to build an $a$-harmonic function that vanishes on $\mathcal{P}$ and is comparable to $U_{a,*}$ and $U_a$.}
Thus, after rescaling, since $0 < \l \leq 1$ was arbitrary, we find that 
\[
|u - p'U_a| \leq |p'U_a - p' U_{a,*}| + |u-p'U_{a,*}| \leq C U_a|X|^\a,
\]
as desired.
\end{proof}

We now proceed with the higher order boundary Harnack estimate. Let $U \in C(B_1)$ be even in $y$ with $U \equiv 0$ on $\P$ and $U>0 $ on $B_1 \setminus \P$, normalized so that $U(\e_n/2) = 1$, and satisfy
\begin{equation}
\label{eqn: La U HOBH C1a}
\La U =|y|^a\bigg(t\frac{U_a}{r} + G(X)\bigg) \quad\text{in }B_1 \setminus \mathcal{P}
\end{equation}
where $t$ is a constant with
\begin{equation}
\label{eqn: La U HOBH C1a RHS}
{|t|} \leq 1 \qquad\text{and}\qquad |G(X)| \leq r^{s-1}|X|^{\a}.
\end{equation}
If $\G \in C^{1, \a}$ with $\|\G\|_{C^{1,\a}} \leq 1$, then Proposition~\ref{prop: C1a ptwise Schauder} implies that
\begin{equation}
\label{eqn: U from schauder C1a}
U = U_a (p' + O(|X|^{\a}))
\end{equation}
for a constant $|p'| \leq C$. As before, formally differentiating \eqref{eqn: U from schauder C1a} yields
\begin{equation}
\label{eqn: grad x C1a}
|\nab_x U  -  p' \na_x U_a| \leq C\frac{U_a}{r}|X|^{\a}
\end{equation}
and
\begin{equation}
\label{eqn: partial y C1a}
| \partial_y U - p' \partial_y U_a| \leq C|y|^{-a}r^{-s}|X|^{\a};
\end{equation}
the justification of these derivative estimates for our application is somewhat delicate and is given in Proposition~\ref{lem: derivative est C1a}.
Again, in the simplest case, taking $t\equiv G \equiv 0$, these derivative estimates can be shown by arguing as in Section~5 of \cite{DS3} and the Appendix of \cite{DS5}, using the regularity results in Section~\ref{sec: Prelim} (cf. Remark~\ref{rmk: ahar case}).

\begin{proposition}
\label{prop: HOBH 1 pos}
Let $\G \in C^{1, \a}$ with $\|\G\|_{C^{1,\a}} \leq 1$. 
Let $U, t, G$, and $p'$ be as in \eqref{eqn: La U HOBH C1a}, \eqref{eqn: La U HOBH C1a RHS}, \eqref{eqn: grad x C1a}, and \eqref{eqn: partial y C1a}.
Suppose that $u \in C(B_1)$ is even in $y$ with $\|u\|_{L^\infty(B_1)} \leq 1$, vanishes on $\mathcal{P}$, and satisfies
\begin{equation*}
\label{eqn: La u HOBH C1a}
\La u =|y|^a\bigg(b\frac{U_a}{r} + F(X)\bigg) \quad\text{in }B_1 \setminus \mathcal{P}
\end{equation*}
where $b$ is a constant such that $|b| \leq 1$ and 
\[
|F(X)| \leq r^{s-1}|X|^{\a}.
\]
Then, there exists a polynomial $P(x,r)$ of degree $1$ with $\| P \| \leq C$ such that 
\[
\bigg| \frac{u}{U} - P \bigg| \leq C|X|^{1+\a}
\]
for some constant $C = C(a,n,\a) > 0$.
\end{proposition}

To prove Proposition~\ref{prop: HOBH 1 pos},
we extend the notion of approximating polynomial to this low regularity setting by considering polynomials in $(x,r_*)$ rather than in $(x,r)$; that is, $P(x,r_*) = p_0 + p_i x_i + p_{n+1}r_* $.
After performing an initial dilation, as before, using Lemma~\ref{lem: Approximations}, \eqref{eqn: grad x C1a}, and \eqref{eqn: partial y C1a}, one can show that
\[
\La (UP) = |y|^a \frac{U_a}{r}\Big(tp_0 + 2s p_n + 2p_{n+1}  \Big) + h(X)
\]
with
\[
|h(X)| \leq \vep \|P\|r^{s-1}|X|^\a.
\]
\begin{definition}
Let $u$ and $U$ be as in Proposition~\ref{prop: HOBH 1 pos}.
A polynomial $P(x,r_*)$ of degree $1$ is {\it approximating} for $u/U$ if 
\[
b =  t p_0 + 2s p_n + 2p_{n+1} .
\]
\end{definition}
With this definition of approximating polynomial,
the proof of Proposition~\ref{prop: HOBH 1 pos} is identical to that of Proposition~\ref{prop: HOBH k pos} and is therefore omitted.

%~~~~~~~~~~~~~~~~~~~~~~~~~~~~~~~~~~~~~~~~~~~~~~~~~~~~~~~%
\section{Proof of Proposition~\ref{prop: Schauder flat boundary}}
\label{sec: Flat}

In this section, we prove Proposition~\ref{prop: Schauder flat boundary}.
That is, if $u$ is $a$-harmonic in $B_1 \setminus \{ x_n \leq 0,\, y = 0 \}$ and vanishes continuously on $\{ x_n \leq 0,\, y = 0 \}$, then the quotient $u/U_a$ is $C^{\infty}_{x,r}(\G \cap B_{1/2})$.
The perturbative arguments of Sections~\ref{sec: HOBH} through \ref{sec: C1a} all rely on this core regularity result.
The idea of the proof is the following. 
The domain $B_1 \setminus \{ x_n \leq 0,\, y = 0 \}$ and the operator $\La$ are translation invariant in the $\e_i$ direction for any $i=1,\dots n-1$, so differentiating the equation $\La u = 0$ shows that $u$ is smooth in these directions. 
We can then reduce the proof of Proposition~\ref{prop: Schauder flat boundary} to the two-dimensional case, but with a right-hand side. 
A final reduction (Lemma~\ref{lem: 2 dim RHS}) leaves us with the main task of this section, which is proving Proposition~\ref{prop: Schauder flat boundary} in the case $n=2$ with zero right-hand side. 
This is Proposition~\ref{prop: 2d flat zipped} below.

It will be convenient to fix the following additional notation.
For $x' \in \R^{n-1}$, we let 
\[
D_{\l,x'} := \{ (x', z) \in \R^{n+1} : |z| < \l \}.
\]
We sometimes suppress the dependence on $x'$ and view $D_{\l ,x'} = D_\l$ as a subset of $\R^2$.

%We begin by proving the following proposition, which is Proposition~\ref{prop: Schauder flat boundary} in dimension two.

\begin{proposition}
\label{prop: 2d flat zipped}
Let $u \in C(B_1)$ be even in $y$ with $\|u\|_{L^{\infty}(D_1)}\leq 1$ and satisfy
\begin{equation*}
\label{eqn: 2 dim RHS zero}
\begin{cases}
\La u =  0 & \text{in }  D_1 \setminus \{x \leq 0,\, y = 0 \} \\
u = 0 & \text{on } \{x \leq 0,\, y = 0 \}.
\end{cases}
\end{equation*}
Then, for any $k \geq 0$, there exists a polynomial $P(x,r)$ of degree $k$ with $\|P\| \leq C$ such that $U_aP$ is $a$-harmonic in $D_1 \setminus \{x \leq 0,\, y = 0 \}$ and 
\[
| u - U_a P | \leq  C U_a |z|^{k+1}
\]
for some constant $C = C(a,k) >0$.
\end{proposition}

The geometry of our domain $D_1 \setminus \{x \leq 0,\, y = 0 \}$ is simplified through the change of coordinates
\[
x(z_1, z_2) = z_1^2 - z_2^2 \qquad\text{and}\qquad y(z_1, z_2) = 2z_1 z_2,
\]
which identifies the right-half unit disk $D_1^{+} := \{ z \in \R^2 : |z| < 1,\, z_1 > 0 \}$ and $D_1 \setminus \{x \leq 0,\, y = 0 \}$.\footnote{\,This can be seen as the complex change of coordinates $z \mapsto z^{2}$, i.e., $\bar{u}(z) = u(z^2)$.
Abusing notation, we let $z$ denote points in this new coordinate system: $z = (z_1,z_2)$.
Similarly, we set $D_\l := \{ z \in \R^2 : |z| < \l \}$.}
If we let $\bar{u}$ denote $u$ after this change of coordinates, then $\bar{u}$ is even in $z_2$ and 
\begin{equation*}
\label{eqn: bLa}
\begin{cases}
\bLa \bar{u} = 0 &\text{in }D_{1}^{+} \\
\bar{u} = 0 &\text{on }\{ z_1 = 0 \}
\end{cases}
\end{equation*}
where the operator $\bLa$ (which is $\La$ in these coordinates) is given by
\[
\bLa u := \frac{1}{4|z|^2}\DIV(|2z_1 z_2|^a \nabla u).
\]
The odd extension of $\bar{u}$ satisfies the same equation in $D_1$.
In this new coordinate system and after an odd extension in $z_1$, the function $U_a$ becomes $|z_1|^{-a}z_1$.
Thus, Proposition~\ref{prop: 2d flat zipped} is equivalent to the following proposition.

\begin{proposition}
\label{prop: 2d flat unzipped}
Let $u \in C(D_1)$ be odd in $z_1$ and even in $z_2$ with $\|u\|_{L^{\infty}(D_1)} \leq 1$ and satisfy
\begin{equation}
\label{eqn: 2d unzipped odd}
\begin{cases} 
\bLa u = 0 &\text{in } D_1 \\
u = 0 &\text{on } \{ z_1 = 0 \}.
\end{cases}
\end{equation}
For any $k \geq 0$, there exists a polynomial $Q$ of degree $2k$ with $\|Q\| \leq C$ such that $\bLa (|z_1|^{-a}z_1 Q) = 0$ in $D_1$ and
\[
|u -  |z_1|^{-a}z_1 Q| \leq C|z_1|^{2s}|z|^{2k+2}
\]
for some constant $C = C(a,k) > 0$.
\end{proposition}

If $a=0$, then $u$ is harmonic and Proposition~\ref{prop: 2d flat unzipped} follows easily. 
Instead, when $a\neq 0$, we prove the result from scratch in three steps.
First, we construct a set homogeneous solutions of \eqref{eqn: 2d unzipped odd}.
Second, we show that these homogeneous solutions form an orthonormal basis for $L^2(\pa D_1)$ with an appropriate weight.
Third, we expand $u|_{\pa D_1}$ in this basis, extend this expansion to the interior of $D_1$, and compare $u$ to the extension.

Let $\bar{\om}_a(z_1,z_2) := |2z_1 z_2|^a$, and observe that $\bar{\om}_a$ is an $A_2$-Muckenhoupt weight.

\begin{remark}[Homogeneous Solutions]
\label{rem: homogeneous solns}
For every $j \in \mathbb{N} \cup \{0\}$, define the function
\[
\bar{u}_j(z_1, z_2) := |z_1|^{-a}z_1 \bar{Q}_j(z_1^2, z_2^2).
\]
Here, $\bar{Q}_j(z_1, z_2) := b_i z_1^{i} z_2^{j-i}$,
\[
b_i := - \frac{(j-i+1)(j-i+1-s )}{i(i+s)}\,b_{i-1},
\]
and $b_0 = b_0(j,a)$ is chosen so that $\|\bar{u}_j\|_{L^2(\partial D_1, \bar{\om}_a)} = 1$.
Each $\bar{u}_j$ is odd in $z_1$, even in $z_2$, vanishes on the $z_2$-axis, and satisfies 
\[
\bLa \bar{u}_j = 0 \quad\text{in } \R^2.
\]
\end{remark}

The two Green's identities below will be used in what follows.
The first is applied to prove Lemma~\ref{lem: monotone}, an important estimate for the proof of Proposition~\ref{prop: 2d flat unzipped}.
The second is utilized in Lemma~\ref{lem: basis} to show that the homogeneous solutions of Remark~\ref{rem: homogeneous solns} form an orthonormal basis for $L^2(\pa D_1,\bar{\om}_{a})$.

\begin{remark}[Green's Identities]
If $u,v \in H^1(D_1,\bar{\om}_{a})$ are such that $\bLa u = \bLa v = 0$ in $D_1$, then $u$ and $v$ satisfy the following Green's identities for $\bLa$:
\begin{equation}
\label{eqn: Green's identity 1}
\int_{D_\l} \na  v \cdot \na u \,\bar{\om}_{a} \, \d z = \int_{\pa D_\l}  u \,\pa_{\nu} v\, \bar{\om}_a \, \d \sig \qquad\forall \l <1
\end{equation}
and
\begin{equation}
\label{eqn: Green's identity 2}
\int_{\pa D_\l} u \, \pa_{\nu} v\, \bar{\om}_a \, \d \sig -
\int_{\pa D_\l} v\, \pa_{\nu} u\, \bar{\om}_a \, \d \sig =0 \qquad\forall \l <1
\end{equation}
\end{remark}

The following lemma shows that the (weighted) $L^2$-norm of $u$ on the boundary of $D_1$ controls the (weighted) $L^2$-norm of $u$ inside $D_1$. 

\begin{lemma}
\label{lem: monotone}
Suppose $u \in H^1(D_1,\bar{\om}_{a})$ satisfies $\bLa u = 0$ in $D_1$.
Then, there exists a positive constant $C$, depending only on $a$, such that
\[
\|u\|_{L^2(D_1, \bar{\om}_a)} \leq C \|u\|_{L^2(\partial D_1, \bar{\om}_a)}.
\]
\end{lemma}

\begin{proof}
Let $\phi$ defined by
\begin{equation*}
% \label{def: phi}
\phi(\l,u) := \fint_{\partial D_\l} |u|^2 \, \bar{\om}_a \, \d \sig,
\end{equation*}
where the average is taken with respect to $\bar{\om}_a$.
Using \eqref{eqn: Green's identity 1}, we compute that $\phi$ is increasing in $\l$.
Hence,
\[
\int_{D_1} |u|^2 \, \bar{\om}_a \, \d z \leq \phi(1,u)\int_0^1 \int_{\partial D_\l} \bar{\omega}_a \, \d \sig \, \d \l = C \int_{\partial D_1} |u|^2 \, \bar{\om}_a \, \d\sig,
\]
as desired.
\end{proof}

Now, let us demonstrate that the homogeneous solutions of Remark~\ref{rem: homogeneous solns} are an orthonormal basis for $L^2(\pa D_1,\bar{\om}_{a})$.

\begin{lemma}
\label{lem: basis}
Let $\iota_a := \|\bar{\om}_a\|_{L^1(\partial D_1)}^{-1}$ and $\bar{u}_j$ be as in Remark~\ref{rem: homogeneous solns}.
The set $\{ \iota_a, \bar u_j : j = 0, 1,\dots \}$ is an orthonormal basis for $L^{2}(\partial D_1,\bar{\om}_{a})$.
\end{lemma}

\begin{proof}
By construction, $\|\iota_a\|_{L^2(\pa D_1,\bar{\om}_{a})} = \| \bar{u}_j \|_{L^2(\pa D_1,\bar{\om}_{a})}=1$.
We show that $\{ \iota_a, \bar u_j : j = 0, 1,\dots \}$ is an orthogonal, dense set in $L^{2}(\partial D_1,\bar{\om}_{a})$.

We first treat the question of density. 
By symmetry, letting  $A := \partial D_1 \cap \{ z_1,z_2 \geq 0 \}$, it suffices to show that $\Span \{ \iota_a, \bar{u}_j: j = 0, 1, \dots \}$ is dense in $L^{2}(A,\bar{\om}_{a})$.
Furthermore, via the locally Lipschitz change of variables $\Phi : [0,1] \to A$ given by $\Phi(t) := (t, \sqrt{1-t^2})$ with Jacobian $J_{\Phi}(t) = (1-t^2)^{-1/2}$, this reduces to showing that $\Span \{ 1 , \bar{w}_j: j = 0, 1, \dots \}$ is dense in $L^{2}([0,1],\bar{\mu}_{a})$ where $\bar{w}_j : = \bar{u}_j \circ \Phi$ and $\bar{\mu}_a := (\bar{\om}_a \circ \Phi) J_{\Phi}$.
To this end, observe that for every $l \in \mathbb{N} \cup \{0\}$, there exist constants $c_j \in \R$ for $j = 0,1,\dots,l$ such that
\[
t^{2s+ 2l} = \sum_{j=0}^{l} c_j \bar{w}_j(t).
\]
The M\"untz-Sz\'asz theorem\footnote{\,The M\"untz-Sz\'asz theorem (see \cite{M} and \cite{Sz}) says that $\Span\{1, t^{p_i} : p_i > 0 \}$ for $i = 0, 1, \dots$ is dense in $C([0,1])$ provided that 
\[
\sum_{i = 0}^{\infty} \frac{1}{p_i} = \infty.
\]
} 
implies that the family $\Span \{1, t^{2s+2l} : l = 0,1, \dots\}$ is dense in $C([0,1])$.
Then, since $C([0,1])$ is dense in $L^2([0,1], \bar{\mu}_a)$, the question of density is settled.

We now show the set $\{ \iota_a, \bar{u}_j : j = 0, 1, \dots \}$ is orthogonal in $L^{2}(\partial D_1,\bar{\om}_{a})$.
First, since the functions $\bar{u}_j$ are odd in $z_1$ and $\bar{\om}_a$ is even, the inner product of $\bar{u}_j$ and $\iota_a$ in $L^{2}(\partial D_1,\bar{\om}_{a})$ is zero for every $j \in \N \cup \{ 0 \}$. 
Moreover, since $\bar{u}_j$ is homogeneous of degree $2s+2j$, one computes for any $z \in \pa D_1$,
\begin{equation*}
\begin{split}
\pa_\nu \bar{u}_j (z) &
= \frac{\d}{\d \l}\bigg{|}_{\l=1} \bar{u}_j(\l z) = (2s + 2j) \bar{u}_j(z).
\end{split}
\end{equation*}
Therefore, for $\bar{u}_k$ and $\bar{u}_j$, \eqref{eqn: Green's identity 2} becomes
\begin{equation*}
0 =(2k-2j) \int_{\pa D_1} \bar u_j \bar u_k \, \bar{\om}_a \, \d \sig,
\end{equation*}
as desired.
\end{proof}

We now prove Proposition~\ref{prop: 2d flat unzipped} (and thus, Proposition~\ref{prop: 2d flat zipped}) using Lemmas~\ref{lem: monotone} and \ref{lem: basis} and the local boundedness property for $\bar{L}_a$ (cf. \eqref{eqn: weak Harnack}).

\begin{proof}[Proof of Proposition~\ref{prop: 2d flat unzipped}]
By Lemma~\ref{lem: basis},
\[
u|_{\pa D_1} =c_a\iota_a +\sum_{j=0}^{\infty} c_j \bar{u}_j |_{\pa D_1}
\]
as a function in $L^2(\partial D_1, \bar{\om}_a)$, and $c_a = 0$ as $u$ is odd in $z_1$.

For every $k \in \N \cup \{0\}$, let $v_k :=  \sum_{j=0}^{k}  c_j \bar{u}_j$ and note that $\bLa v_k = 0$ in $D_1$.
So, applying Lemma~\ref{lem: monotone} and \cite[Corollary~2.3.4]{FKS} to $w_{k+1} := u - v_k$,
\[
\lim_{k\to \infty} \|u-v_k\|_{L^\infty(D_{1/2})} \leq C\lim_{k\to \infty} \|u - v_k\|_{L^2(D_1, \bar{\om}_a)} \leq \lim_{k\to \infty} C \|u - v_k\|_{L^2(\partial D_1, \bar{\om}_a)} = 0.
\]
Therefore,
\[ 
u = \sum_{j=0}^{\infty}  c_j \bar{u}_j \quad\text{in } D_{1/2}.
\]
Since $\|u\|_{L^{\infty}(D_1)} \leq 1$, it follows that $\sup_j |c_j| \leq 1$.
Consequently, 
\[
|w_{k+1}| \leq |u| + |v_k| \leq C
\]
for a constant $C$ depending only on $k$ and $a$.
For any $0<\l<1$, define the rescaling
\[
\tilde{w}(z) := \frac{w_{k+1}(\l z)}{\l^{2s+2k+2}},
\]
and observe that $|\tilde{w}|\leq C$ in $D_{1/2}$ thanks to the homogeneity of each term in $w_{k+1}$. 
Furthermore, the functions $\tilde{w}$ and $|z_1|^{-a}z_1$ vanish on $\{z_1 = 0\}$ and satisfy $\bLa \tilde{w} = \bLa (|z_1|^{-a}z_1) = 0$ in $D_{1/2}$. 
Applying the boundary Harnack estimate in $D_{1/2} \cap \{z_1 >0\}$ and recalling that $\tilde{w}$ and $|z_1|^{-a}z_1$ are odd in $z_1$, we find that 
\begin{equation*}
\label{eqn: BH bound}
|\tilde{w} | \leq C |z_1|^{2s} \quad\text{in }D_{1/4}.
\end{equation*}
Rescaling and letting $Q(z) := \sum_{j=0}^{k} c_j \bar{Q}_j(z_1^2,z_2^2)$, we deduce that
\[
|u -  |z_1|^{-a}z_1 Q| \leq C|z_1|^{2s}|z|^{2k+2},
\]
concluding the proof.
\end{proof}

With Proposition~\ref{prop: 2d flat zipped} in hand, we use a perturbative argument to prove the following, which in turn will allow us to conclude the proof of Proposition~\ref{prop: Schauder flat boundary}.

\begin{lemma}
\label{lem: 2 dim RHS} 
Suppose $u\in C(B_1)$ is even in $y$ with $\|u\|_{L^{\infty}(D_1)}\leq 1$, vanishes on $\{x \leq 0,\, y = 0 \}$, and satisfies
\begin{equation*}
\label{eqn: 2 dim RHS}
\La u(z) = | y|^a\bigg(\frac{U_a}{r}R(x,r)+ F(z)\bigg) \quad\text{in } D_1 \setminus \{x \leq 0,\, y = 0 \}
\end{equation*}
where $R(x,r)$ is a polynomial of degree $k$ with $\|R\| \leq 1$ and 
\begin{equation*}
\label{eqn: 2 dim remainder}
|F(z)| \leq r^{s-1}|z|^{k + \alpha}.
\end{equation*} 
Then, there exists a polynomial $P(x,r)$ of degree $k+1$ such that $\|P\| \leq C$ and
\begin{equation*}
| u - U_a P | \leq C U_a |z|^{k+1+\a}
\end{equation*}
for some constant $C = C(a,k,\a) >0$.
\end{lemma}

The proof of Lemma~\ref{lem: 2 dim RHS} follows the two step improvement of flatness and iteration procedure given in the proof of Proposition~\ref{prop: HOBH k pos},
and so we omit it.\footnote{\,Since $U_a$ and $r$ are homogeneous when $\G$ is flat and $U_a$ is $a$-harmonic away from the set $\{x_n \leq 0 ,\, y = 0\}$, the notion of approximating polynomials and proof are much simpler.}

We now prove Proposition~\ref{prop: Schauder flat boundary}.
\begin{proof}[Proof of Proposition~\ref{prop: Schauder flat boundary}]
The equation \eqref{eqn: G flat eq} is invariant with respect to $x'$, so any partial derivative of $u$ in an $x'$-direction also satisfies \eqref{eqn: G flat eq}. 
Furthermore, the set $\{x_n \leq 0,\, y = 0 \}$, where $u$ vanishes, has uniformly positive $\La$-capacity.
It follows that solutions of \eqref{eqn: G flat eq} are uniformly H\"older continuous in compact subsets of $B_1$.
In particular, as $\|u\|_{L^{\infty}(B_1)} \leq 1$, we see that for any multi-index $\mu$,
\begin{equation}
\label{eqn: La cap}
\|D^\mu_{x'} u\|_{C^{0,\tau}(B_{1/2})} \leq C
\end{equation}
for some constant $C = C(a,|\mu|) > 0$; that is, $u \in C^{\infty}_{x'}(B_{1/2})$.

With the regularity of $u$ in $x'$ understood, we turn to understanding the regularity of $u$ in $z$.
To this end, notice that 
\[
\La u  =|y|^a\Delta_{x'} u + \DIV_{z}(|y|^a\nab_{z}u).
\]
So, for any fixed $x'$, $u$ satisfies
\[
\begin{cases}
\La u = -|y|^a \Delta_{x'} u & \text{in } D_{1,x'} \setminus \{x_n \leq 0,\, y = 0\}\\
u=0 & \text{on } \{x_n \leq 0 ,\, y=0\}
\end{cases}
\]
as a function of $z$. 
(Here, $\La$ is seen as a two-dimensional operator.)
Let
\[
f := - \Delta_{x'} u.
\]
By \eqref{eqn: La cap}, $f$ is $C^{0,\tau}$ in compact subsets of $B_1$.
Up to multiplying by a constant, we may assume that $\|f\|_{C^{0,\tau}(B_{1/2})} \leq 1$.
Then, since $f(x', 0) = 0$ for every $x'$, viewed just as a function in $z$, we see that
\[
|f(z)| \leq  |z|^{\tau} \leq  r^{s-1} |z|^{\tau}.
\]
In particular, $u$ satisfies the hypotheses of Lemma~\ref{lem: 2 dim RHS} with $k=0$ taking $R(x_n,r) \equiv 0$ and $F(z) = f(z)$.
Applying Lemma~\ref{lem: 2 dim RHS}, we find a degree 1 polynomial $P_0(x_n,r)$ such that 
\[
| u - U_a P_0|  \leq  C U_a |z|^{1+\tau}.
\]
As $u$ and $f$ have the same regularity in $z$, it follows that
\[ 
f(z) = U_a Q_0(x_n,r) + U_a O(|z|^{1+\tau})
\] 
for a polynomial $Q_0(x_n,r)$ of degree $1$. 
Equivalently,
\[
f(z) = \frac{U_a}{r}R(x_n,r) + F(z)
\]
where $R(x_n,r) = rQ_0(x_n,r)$ is a degree $2$ polynomial and, up to multiplication by a constant, $\|R\| \leq 1$ and $|F(z)| \leq r^{s-1}|z|^{2+\tau}$.
In this way, we can bootstrap Lemma~\ref{lem: 2 dim RHS} with $k = 2j$ to find polynomials $P_{j}(x_n,r)$ of degree $2j+1$ such that
\[
|u - U_aP_{j}| \leq CU_a|z|^{2j+1+\tau}.
\]
In other words, for each $x'$, there exists $\phi_{x'} \in C^\infty(D_{1/2, x'})$ such that 
\[ 
u(x',z) = U_a(z) \phi_{x'}(z).
\]

Since $u$ is smooth in $x'$ and $U_a$ is independent of $x'$, $\phi_{x'}(z) = \Phi(x',z)$ is a smooth function of $x'$ as well.
So,
\[
u(x',z) = U_a(z) \Phi(x',z).
\]
This proves \eqref{eqn: flat implies smooth}.

Finally, that $U_aP$ is $a$-harmonic in $B_1 \setminus \mathcal{P}$ follows by induction: decompose $P$ into its homogeneous parts $P(x,r) = \sum_{m =0}^k P_m(x,r)$, where each $P_m(x,r)$ is a homogeneous polynomial of degree $m$.
If $m = 0$, then $\La (U_aP) = 0$ in $B_1 \setminus \{x_n \leq 0,\, y = 0 \}$.
Assuming $U_a\sum_{m =0}^l P_m(x,r)$ is $a$-harmonic in $B_1 \setminus \{x_n \leq 0,\, y = 0 \}$ for all $l < k$,
we find that
\begin{equation}
\label{eqn: induction expansion}
v := u - U_a\sum_{m =0}^l P_m(x,r) = U_a\big( P_{l+1}(x,r) + o(|X|^{l+1})\big)
\end{equation}
is $a$-harmonic in $B_1 \setminus \{x_n \leq 0,\, y = 0 \}$.
Now, consider the rescalings of $v$ defined by
\[
\tilde{v}(X) := \frac{v(\l X)}{\l^{l+1+s}}.
\]
Observe that $\La \tilde{v} = 0$ in $B_1 \setminus \{x_n \leq 0,\, y = 0 \}$ and that $\tilde{v} \to U_aP_{l+1}$ as $\l \to 0$ by \eqref{eqn: induction expansion}.
Therefore, $\La (U_aP_{l+1}) = 0$ in $B_1 \setminus \{x_n \leq 0,\, y = 0 \}$, which concludes the proof.
\end{proof}

%~~~~~~~~~~~~~~~~~~~~~~~~~~~~~~~~~~~~~~~~~~~~~~~~~~~~~~~%
\section{Proof of Theorem~\ref{thm: main}}
\label{sec: Proof of Main Thm}

In this section, we prove Theorem~\ref{thm: main}.
Up to multiplication by a constant, we may assume that 
\[
\| \vphi \|_{C^{m,\beta}(\R^n)} \leq 1 \qquad\text{and}\qquad \|v - \vphi \|_{L^\infty(\R^n)} \leq 1;
\] 
and after a translation, rotation, and dilation, we can assume that the origin is a regular point of $\G$, that $\G \cap B_2^*$ can be written as the graph of a $C^{1,\sigma}$ function of $n-1$ variables, and that $\pa_{\R^n} \P = \G$ is locally given by 
\[
\G  = \{ (x',\gamma(x'),0) :  x'\in B_1' \}
\]
where $\gamma : B'_1 \to \R$ is such that
\[
\gamma(0) = 0, \qquad \nab_{x'}\gamma(0) = 0, \qquad\text{and}\qquad \|\gamma\|_{C^{1,\alpha}(B'_1)} \leq 1.
\] 
Here, $\a := \min\{\beta,\sig\}$. 
Now, recalling the discussion following the statement of Theorem~\ref{thm: DS5 HOBH}, our goal is to apply Proposition~\ref{prop: HOBH 1 pos} and then iteratively apply Proposition~\ref{prop: HOBH k pos} to produce a polynomial $P$ of degree $m-2$ for which, after restricting to the hyperplane $\{y=0\}$, we have
\begin{equation*}
\bigg| \frac{\pa_i(v-\vphi)}{ \pa_n(v - \vphi) }-P \bigg| \leq C|x|^{m-2+\a}.
\end{equation*} 

The functions $\pa_i(v-\vphi)$ and $\pa_n(v - \vphi)$ are only defined on $\R^n$, so, to apply our higher order boundary Harnack estimates, we must first extend $v$ and $\vphi$ to $\R^{n+1}$.
Following the notation set in the introduction, we denote  the (even in $y$) $a$-harmonic extension of our solution $v$ by $\tilde{v}$, which satisfies \eqref{eqn: obstacle problem local}.
Choosing an extension $\tilde{\vphi}$ for the obstacle to $\R^{n+1}$ is less straightforward, as our choice governed by the need for the pair $u = \pa_i(\tilde{v} - \tilde{\vphi})$ and $U = \pa_n(\tilde{v} - \tilde{\vphi})$ to satisfy the hypotheses of Propositions~\ref{prop: HOBH 1 pos} and \ref{prop: HOBH k pos}.
The primary challenge is to show that $U$ will satisfy the derivative estimates
%most challenging of the hypotheses and the foci of this section are 
\eqref{eqn: grad x C1a}, \eqref{eqn: partial y C1a}, \eqref{eqn: grad x}, and \eqref{eqn: grad r}.
With this in mind, we define
\begin{equation}\label{eqn: expansion}
\tilde{\vphi}(X) := \vphi(x) + \sum_{j = 1}^{\lfloor\frac{m}{2}\rfloor+ 1} \frac{(-1)^{j}}{c_j} y^{2j}\Delta^j T^{0}(x)
\end{equation}
where $c_0 := 1$, $c_j := 2j(2j+a-1)c_{j-1}$, and $T^{0} = T^0(x)$ is the $m$th order Taylor polynomial of $\vphi$ at the origin.
The coefficients $c_j$ are chosen such that 
\[
\La \tilde{\vphi}(X) = |y|^a\Delta (\vphi - T^0)(x).
\]
Set
\begin{equation}
\label{eqn: tilde u 0}
w(X) := \tilde{v}(X) - \tilde{\vphi}(X),
\end{equation}
For any $i \in 1,\dots, n$, \cite[Proposition 4.3]{CSS} implies that $\pa_i w \in C^{0,\delta}(B_{1})$ for all $0<\delta<s$, and, up to multiplication by universal constant, $\| \pa_i w\|_{C^{0,\delta}(B_{1})} \leq 1$. 
So, $\pa_i w$ satisfies
\begin{equation}\label{eqn: equation}
\begin{cases}
\La \partial_i w = |y|^a f_i &\text{in } B_1 \setminus \P \\
\partial_i w = 0 &\text{on } \P
\end{cases}
\qquad\text{and}\qquad
f_i := \Delta (\pa_i \vphi - \pa_i T^{0}).
\end{equation}
By construction, $f_i = f_i(x)$ is of class $C^{m-3,\b}$ and
\begin{equation}\label{eqn: bound}
|f_i| \leq |x|^{m-3+\b}.
\end{equation}
Again, in order to apply the higher order boundary Harnack estimates we must justify \eqref{eqn: grad x C1a}, \eqref{eqn: partial y C1a}, \eqref{eqn: grad x}, and \eqref{eqn: grad r} for $U = \pa_n w$.
We do this in two propositions.
The following addresses \eqref{eqn: grad x} and \eqref{eqn: grad r}, while Proposition~\ref{lem: derivative est C1a} below addresses \eqref{eqn: grad x C1a} and \eqref{eqn: partial y C1a}.

\begin{proposition}
\label{prop: deriv estimates} 
Let $\G \in C^{k+2, \a}$ with $ 0 \leq k\leq m-3$ and $\|\G\|_{C^{k+2,\a}}\leq 1$.
Let $U:=\pa_n w$ for $w$ as defined in \eqref{eqn: tilde u 0}.
Let $P_0(x,r)$ be the polynomial of degree $k+1$ obtained from Proposition~\ref{prop: pointwise Schauder Cka}.
Then,
\begin{equation}
\label{eqn: partial in i}
\na_x U  = \frac{U_a}{r}\Big( s P_0 \nu + r \na_x P_0  + (\pa_r P_0 ) d \nu + O(|X|^{k+1+\a})\Big)
\end{equation}
and
\begin{equation}
\label{eqn: partial in r}
\nab U \cdot \nab r = \frac{U_a}{r} \Big( sP_0 + (\pa_r P_0) r + \na_x P_0 \cdot (d\nu) + O(|X|^{k+1+\a})\Big).
\end{equation}	
\end{proposition}

As we shall see, to prove Proposition~\ref{prop: deriv estimates}, we stitch together a family of analogous estimates in overlapping cones based at points on $\G$ in a neighborhood of the origin.
In the cone 
\begin{equation}
\label{def: cone}
\mathcal{K} := \{ |z| > |x'| \},
\end{equation}
these estimates are given by the following lemma. 

\begin{lemma}
\label{lem: derivative est}
Fix $m \geq 4$ and $ 0 \leq k\leq m-3$. 
Let $\G \in C^{k+2, \a}$ with $\|\G\|_{C^{k+2,\a}}\leq 1/4$.  
Suppose $U\in C(B_1)$ is even in $y$ with $\|U\|_{L^\infty(B_1)} \leq 1$ and
\[
\begin{cases} L_a U = |y|^a f & \text{in } B_1 \setminus \mathcal{P}\\
U = 0 & \text{on } \mathcal{P}
\end{cases}
\]
where $f = f(x)$ and $f \in C^{m-3,\a}(B_1)$ is such that  $\| f \|_{C^{m-3,\a}(B_1)} \leq 1$ with vanishing derivatives up to order $m-3$ at the origin.
Let $P = P_0$ be the polynomial of degree $k+1$ obtained in Proposition~\ref{prop: pointwise Schauder Cka}. 
Then,
\begin{equation}
\label{eqn: partial in x_i expansion}
|\partial_i U -  \pa_i(U_a P)| \leq C\frac{U_a}{r}|X|^{k+1+\a} \quad \text{in } \mathcal{K}
\end{equation}
and
\begin{equation}
\label{eqn: partial in y expansion}
| \partial_y U - \partial_y(U_aP)| \leq C|y|^{-a}|X|^{k+1+\a-s} \quad\text{in }\mathcal{K}\end{equation}
for some constant $C = C(a,n,k,\a) > 0$.
% As a consequence, \eqref{eqn: partial in i} and \eqref{eqn: partial in r} hold for $U$ in $\mathcal{K}$.
%Here, $P_i(x,r)$ is obtained by formally differentiating $U_aP$ in the $x_i$-direction.
\end{lemma}

\begin{proof}
Since $|\La U| \leq |y|^a r^{s-1}|X|^{k+\a}$, Proposition~\ref{prop: pointwise Schauder Cka} can indeed be applied to obtain an approximating polynomial $P(x,r)$ for $U/U_a$ of degree $k+1$ with $\| P\| \leq C$ such that
\[
|U-U_aP| \leq CU_a|X|^{k+1+\a}. 
\]
%By assumption, $\G \subset \{ 4|z| \leq |x'| \}$.
For a fixed $0<\l <1$, define
\[
\tilde{U}(X) := \frac{[U - U_a P](\l X)}{\l^{k+1+\a+s}}.
\]
By construction, $\|\tilde{U}\|_{L^\infty(B_1)} \leq C$.
Furthermore,
\[
\La \tilde{U} = |y|^a F \quad\text{in } B_1 \setminus \mathcal{P}_\l
\]
where, since $P$ is an approximating polynomial for $U/U_a$,
\[
F(X) := \frac{f(\l x)}{\l^{k-1+\a+s}} - \frac{U_{a,\l}}{r_\l}\frac{h(\l X)}{\l^{k+\a}}.
\]
Here, $h(X) = \sum_{l=0}^k r^l h_l(x)$, and $h_l \in C^{k,\alpha}(B_1^*)$ have vanishing derivatives up to order $k-l$ at zero; recall the discussion on approximating polynomials in Section~\ref{sec: Schauder Cka}.
We decompose $\tilde{U} $ as $\tilde{U} = \tilde{U}_1 - \tilde{U}_2$ where $\tilde{U}_i \equiv 0$ on $\P_\l$ and $\La \tilde{U}_i = |y|^aF_i$ with
\[
F_1(x) := \frac{f(\l x)}{\l^{k-1+\a+s}} \qquad\text{and}\qquad F_2(X) := \frac{U_{a,\l}}{r_\l}\frac{h(\l X)}{\l^{k+\a}}. 
\]
Notice that $F_1 = F_1(x)$ is of class $C^{1,\a}$ with $C^{1,\a}$-norm bounded independently of $\l$ and that $F_2 \equiv 0$ on $\P_\l$.\\

{\it Proof of \eqref{eqn: partial in y expansion}.}
Since $\tilde{U}$ vanishes on $\mathcal{P}_\l$, by \eqref{eqn: Ua property}, 
\begin{equation*}
\label{eqn: bound on conjugate RHS}
% \begin{split}
\big||y|^a\partial_y F(X)\big| \leq  
C\frac{(r_\l - d_\l)^{1-s}}{r^2_\l}|X|^{k+\a} + C|y|^{a+1}\frac{U_{a,\l}}{r^3_\l}|X|^{k+\a} \leq C
% \end{split}
\end{equation*}
in $\mathcal{C} \cap (B_{7/8} \setminus \overline{B}_{1/8})$ where $\mathcal{C}  := \{ 2|z| > |x'| \}$.
Hence, Proposition~\ref{prop: derivative estimates} and Corollaries~\ref{cor: basic bd reg} and \ref{cor: y bd reg 2} (applied to $\tilde{U}_2$ and $\tilde{U}_1$ respectively) imply that 
\[
||y|^a\pa_y\tilde{U}| \leq C \quad \text{in }\mathcal{K} \cap (B_{3/4} \setminus \overline{B}_{1/4}).
\]
Expressing this in terms of $U$ and rescaling, \eqref{eqn: partial in y expansion} follows as $\l >0$ was arbitrary.\\

{\it Proof of \eqref{eqn: partial in x_i expansion}.}
Notice that $F$ is bounded independently of $\l$ in the region $\mathcal{C} \cap (B_{7/8} \setminus \overline{B}_{1/8})$.
Since $\tilde{U}$ vanishes on $\mathcal{P}_\l$, Proposition~\ref{prop: derivative estimates} and Corollary~\ref{cor: basic bd reg} imply that
\[
|\partial_i \tilde{U} | \leq C \quad\text{in } \mathcal{K} \cap (B_{3/4} \setminus \overline{B}_{1/4}).
\]
We need to improve this inequality to
\begin{equation}
\label{eqn: improved est}
|\partial_i \tilde{U} | \leq CU_{a,\l} \quad\text{in } \mathcal{K} \cap (B_{3/4} \setminus \overline{B}_{1/4}).
\end{equation} 
Let $\mathcal{K}^{\pm}$ be the upper and lower halves of $\mathcal{K}$ with respect to $x_n$, that is, 
\begin{equation}
\label{def: cone plus minus}
\mathcal{K}^+ := \mathcal{K} \cap \{x_n\geq 0\}\qquad \text{and } \qquad \mathcal{K}^- := \mathcal{K} \cap \{ x_n <0\}.
\end{equation}
In $\mathcal{K}^+ \cap (B_{3/4} \setminus \overline{B}_{1/4})$, we see that $U_{a,\l} \geq c$, 
%because $r_\l > C|x'|$ in $\mathcal{K}$ and $d_\l > 0$ in $\mathcal{K}^+$,
and so \eqref{eqn: improved est} is immediate in this region.
On the other hand, 
% in $\mathcal{K}^- \cap (B_{3/4} \setminus \overline{B}_{1/4})$, the functions $\tilde{U}$ and $F_2$ vanish on $\{ 2|z| > |x'| \} \cap \{ y = 0 \}$ and $|y|^a\pa_y F_2$ is bounded independently of $\l$ in $\{2 |z| > |x'| \} \cap (B_{7/8} \setminus \overline{B}_{1/8})$.
by Corollaries~\ref{cor: upgrade f is C1} and \ref{cor: upgrade f is 0} (applied to $\tilde{U}_1$ and $\tilde{U}_2$ respectively), \eqref{eqn: improved est} holds in $\mathcal{K}^- \cap (B_{3/4} \setminus \overline{B}_{1/4})$.
Thus, \eqref{eqn: partial in x_i expansion} follows as $\l > 0$ was arbitrary.
\end{proof}

We now prove Proposition~\ref{prop: deriv estimates}. The idea is to define a different extension $\tilde{\vphi}^{x_0}$ of $\vphi$ at every point in $\G \cap B^*_{1/2}$ in such a way that allows us to apply Lemma~\ref{lem: derivative est} to $\pa_n (\tilde{v} - \tilde{\vphi}^{x_0})$ in cones based at $x_0$.
Then, we patch the estimates from Lemma~\ref{lem: derivative est} together and conclude.

\begin{proof}[Proof of Proposition~\ref{prop: deriv estimates}]
For each $x_0 \in \G \cap B_{1/2}^*$, define
\[
\tilde{\vphi}^{x_0}(X) := \vphi(x) + \sum_{j = 1}^{\lfloor\frac{m}{2}\rfloor+ 1} \frac{(-1)^{j}}{c_j} y^{2j}\Delta^j T^{x_0}(x),
\]
with $c_j$ as in \eqref{eqn: expansion} and $T^{x_0}(x)$ the $m$th order Taylor polynomial of $\vphi$ at $x_0$. Set
\begin{equation*}
w^{x_0}(X) := \tilde{v}(X) - \tilde{\vphi}^{x_0}(X).
\end{equation*}
% Notice that $w^{x_0}\equiv 0 $ in $\P$ and $\| w^{x_0}\|_{L^\infty(B_2)} \leq C.$
Letting $X_0 := (x_0,0)\in \R^{n+1}$, 
%For any $i \in 1,\dots, n$ and $x_0 \in \G \cap B_{1/2}^*$, \cite[Proposition 4.3]{CSS} implies that $\pa_i w^{x_0}\in C^{0,\delta}(B_{1}(X_0))$ for all $0<\delta<s$ with $\| \pa_i w^{x_0}\|_{C^{0,\delta}(B_{1}(X_0))} \leq C$.
we again see that $\|\pa_n w^{x_0}\|_{C^{0,\delta}(B_{1}(X_0))} \leq 1$ and
\[
\begin{cases}
\La \partial_n w^{x_0} = |y|^a f^{x_0}_n &\text{in } B_1(X_0) \setminus \P \\
\partial_n w^{x_0} = 0 &\text{on } \P
\end{cases}
\qquad
\text{where}\qquad 
f^{x_0}_n := \Delta (\pa_n \vphi - \pa_n T^{x_0});
\]
and, by construction, $f_n^{x_0} = f_n^{x_0}(x)$ is of class $C^{m-3,\b}$ and $|f^{x_0}_n| \leq |x-x_0|^{m-3+\b}.$
So, up to a dilation, we apply Lemma~\ref{lem: derivative est} to $U = U^{X_0} :=\pa_n w^{x_0}$ for every $x_0 \in \G \cap B_{1/2}^*$ with right-hand side $f = f_n^{x_0}$.
As $\G \in C^{k+2,\a}$, after Taylor expanding $\nu_i$ and $ d $, we have 
\begin{equation}
\label{eqn: new i}
\Big|\pa_i U^{X_0} -  \frac{U_a}{r}P_{X_0}^i\Big| \leq \frac{U_a}{r} C|X-X_0|^{k+1+\a} \quad \text{in } \mathcal{K}_{X_0}
\end{equation}
and
\begin{equation}
\label{eqn: new r}
\Big|\na U^{X_0}\cdot \na r -  \frac{U_a}{r}P^r_{X_0}\Big| \leq \frac{U_a}{r} C|X-X_0|^{k+1+\a}\quad \text{in } \mathcal{K}_{X_0}
\end{equation}
where $P^i_{X_0}$ and $P^r_{X_0}$ are polynomials of degree $k+1$ and $\mathcal{K}_{X_0}$ is the rotation and translation of $\mathcal{K}$ centered at $x_0$ pointing in the direction $\nu(x_0)$.

We now show that \eqref{eqn: new i} and \eqref{eqn: new r} hold for $U^0$ in all of $B_1$.
Given any $X \in B_{1}$, let $X_0 \in \G$ be such that $r(X) = |X-X_0|$; note that $X \in \mathcal{K}_{X_0}$ and
\begin{equation}\label{eqn: cone comparability}
|X-X_0| \leq |X|.
\end{equation}
Then,
\begin{equation*}
\label{eqn: triangle}
\Big| \pa_i U^{0} -\frac{U_a}{r} P^i_{0} \Big| 
\leq \Big| \pa_i U^{X_0} -\frac{U_a}{r} P^i_{X_0} \Big| +\left| \pa_i U^0 -  \pa_i U^{X_0}\right|+\frac{U_a}{r} \left|  P^i_{X_0} - P^i_{0}\right| =
{\rm I} + {\rm II}+{\rm III}.
\end{equation*}
As $X\in \mathcal{K}_{X_0}$,  \eqref{eqn: new i} and \eqref{eqn: cone comparability} imply that
\[
{\rm I} \leq C\frac{U_a}{r}|X|^{k+1+\a}.
\]
Note that 
\[
{\rm II} 
\leq \sum_{j = 1}^{\lfloor\frac{m}{2}\rfloor+ 1} \frac{(-1)^{j}}{c_j} y^{2j}\left| \Delta^j \pa_{in} (T^{x_0}-  T^0 )\right| .
\]
If $2j+2>m$, then $\Delta^j \pa_{in} T^0 = \Delta^j \pa_{in} T^{x_0} = 0$. 
On the other hand, if $ 2j+2 \leq m$, then $\Delta^j \pa_{in} T^0 $ and $\Delta^j \pa_{in} T^{x_0} $ are the Taylor polynomials of degree $m-2j-2$ at zero and $x_0$ respectively for $\Delta^j \pa_{in} \vphi$. 
Therefore, again using \eqref{eqn: cone comparability},
\begin{align*}
|\Delta^j \pa_{in} (T^{x_0}- T^0)| %&\leq |\Delta^j \pa_{in} T^{x_0} (x) - \Delta^j \pa_{in} \vphi(x)| + | \Delta^j \pa_{in} T^0(x) - \Delta^j \pa_{in} \vphi(x)| \\
%& \leq C|X-X_0|^{m-2j-2 +\a} + C|X|^{m-2j-2 +\a}
\leq C|X|^{m-2j-2 +\a}.
\end{align*}
Recalling \eqref{eqn: Ua property} and that $m-2 \geq k+1$, we see that
\begin{align*}
{\rm II} 
\leq C|y|^{2j}|\Delta^j \pa_{in} (T^{x_0}- T^0)|
%&\leq  {\color{cyan}C|y|^{2j}|X|^{m-2j-2 +\a}} \\
%& = {\color{cyan}C\frac{U_a}{r} r(r-d)^s|y|^{2j-2s}|X|^{m-2j-2 +\a}}\\
%& {\color{cyan}\leq C\frac{U_a}{r} |X|^{m-1-s+\a}} 
&\leq C\frac{U_a}{r} |X|^{m-2+\a}
\leq  C\frac{U_a}{r} |X|^{k+1+\a}.
\end{align*}
Finally, to bound III, let $\l :=|X_0|$. 
Since $\G \subset\{|z|<|x'|\}$, the ball $B_{\l/2}(2\l \e_n)$ is contained in $\mathcal{K} \cap \mathcal{K}_{X_0}$. 
Observe that
\[
\frac{U_a}{r} \left|  P^i_{X_0} - P^i_{0}\right| 
\leq \bigg|\pa_i U^{X_0} - \frac{U_a}{r}P^i_{X_0} \bigg|+ \bigg| \pa_i U^0 - \frac{U_a}{r} P^i_{0} \bigg| + \left| \pa_i U^{X_0} - \pa_i U^0 \right|  \quad\text{in } B_{\l/2}(2\l \e_n).
\]
Noting that $|Z|$ and $|Z-X_0|$ are of order $\l$ for all $Z \in B_{\l/2}(2\l \e_n)$, the bound on ${\rm II}$ and \eqref{eqn: new i} imply that
$\|P^i_{X_0} - P^i_{0}\|_{L^\infty(B_{\l/2}(2\l\e_n))} \leq C\l^{k+1+\a}$.
Hence,
\[
\| P^i_{X_0} - P^i_{0}\|_{L^\infty(B_{4\l})} \leq C\l^{k+1 +\a},
\]
and, in particular, we determine that
\[
{\rm III} \leq C\frac{U_a}{r} |X|^{k+1+\a}.
\]
We conclude that \eqref{eqn: partial in i} holds for $U^0$ in $B_{1}$. 

An identical argument shows that \eqref{eqn: partial in r} holds for $U^0$ in $B_1$.
\end{proof}

Now we address the case when $\G \in C^{1,\a}$. The following proposition shows that \eqref{eqn: grad x C1a} and \eqref{eqn: partial y C1a} hold for $U = \pa_n w$, with $w$ defined as in \eqref{eqn: tilde u 0}.
\begin{proposition}
\label{lem: derivative est C1a}
Let $\G \in C^{1,\a}$ with $\| \G \|_{C^{1,\a} } \leq 1$.
Let $U \in C(B_1)$ be even in $y$ and normalized so that $U(\e_n/2) = 1$. Let $U \equiv 0$ on $\mathcal{P}$ and $U > 0$ in $B_1 \setminus \mathcal{P}$, and suppose $U$ satisfies
\begin{equation*}
\label{eqn: La U HOBH C1a grad est}
\La U =|y|^{a}f \quad\text{in }B_1 \setminus \mathcal{P}
\end{equation*}
where $f = f(x)$ and $f \in C^{0,1}(B_1)$ with $\|f\|_{C^{0,1}(B_1)} \leq 1$. 
Let $p'$ be the constant obtained in Proposition~\ref{prop: C1a ptwise Schauder}. 
Then,
\begin{equation*}
\label{eqn: partial in x_i expansion C1a}
|\nab_x U  -   p' \na_x U_a| \leq C\frac{U_a}{r}|X|^{\a}
\end{equation*}
and
\begin{equation*}
\label{eqn: partial in y expansion C1a}
| \partial_y U - p' \partial_yU_a| \leq C|y|^{-a}r^{-s}|X|^{\a}
\end{equation*}
for some constant $C = C(a,n,\a) > 0$.
\end{proposition}

\begin{proof}
Note that Proposition~\ref{prop: C1a ptwise Schauder} can be applied because $|\La U| \leq |y|^ar^{s-2+\a}$.
Let $Z$ be a point of differentiability for $r$ with distance $\l/2$ from $\G$.
Up to a translation, we may assume that the closest point on $\G$ to $Z$ is the origin.
So, at $Z$, we have that
\begin{equation}
\label{eqn: equalities at Z}
U_a = \bar U_a, \qquad r = \l/2, \qquad\text{and}\qquad  \na U_a = \na\bar U_a.
\end{equation}
Set
\[ 
\tilde{U}(X): =\frac{ [U- p'\bar{U}_a](\l X)}{\l^{\a+s}}
\]
and $\mathcal{C} : = \{  2|z| > |x'| \}$, and
let $\mathcal{K}$ and $\mathcal{K}^{\pm}$ be as in \eqref{def: cone} and \eqref{def: cone plus minus}. 
Arguing as in the proof of Lemma~\ref{lem: Approximations}, where we obtained that $|U_{a,*} - U_a| \leq CU_a r^\a$, we see that
\[
|U - p'\bar{U}_a | \leq |U - p' U_a | + |p'||\bar{U}_a - U_a| \leq C\l^{\a+s} \quad\text{in } \mathcal{C}  \cap \{\l/8 \leq |z| \leq 7\l/8 \}.
\]
Thus, $|\tilde{U}| \leq C$ in $ \mathcal{C}  \cap \{\l/8 \leq |z| \leq 7\l/8 \}$ for all $\l > 0$.
Notice that
\begin{equation*}
\label{eqn: La w C1a grad est}
\La \tilde{U} = |y|^aF \quad\text{in } B_1 \setminus \mathcal{P}_\l
\end{equation*}
where, recalling that $\bar{U}_a$ is $a$-harmonic in $B_1 \setminus \{ x_n \leq 0,\, y = 0 \}$,
\[
F(X) = \l^{2-s-\a}f(\l x).
\]
Observe that $F = F(x)$, $F \in C^{0,1}(B_1)$, and $\|F\|_{C^{0,1}(B_1)} \leq 1$.
Arguing as in Lemma~\ref{lem: derivative est} and by \eqref{eqn: equalities at Z}, we find that 
\[
|\pa_y U - p' \pa_y U_a| \leq C r^{\a - s} |y|^{-a}\qquad\text{and}\qquad
| \na_x U -  p' \na_x U_a |  \leq C U_a r^{\a -1} \quad\text{at } Z.
\]
Moreover, since the origin was distinguished by an arbitrary translation,
\[
|\pa_y U - p'_{X_0} \pa_y U_a| \leq C r^{\a - s} |y|^{-a}
\qquad\text{and}\qquad
| \na_x U -  p'_{X_0} \na_x  U_a | \leq C U_a r^{\a -1}
\]
for every $X \in B_{1/2}$ at which $r$ is differentiable, letting $X_0$ be the projection of $X$ onto $\G$ and $p'_{X_0}$ be the constant corresponding to the expansion of $U$ at $X_0$.
Since
\[
| p'_{X_0} - p'| \leq C|X|^{\a},
\]
the lemma follows.
\end{proof}

Finally, we prove Theorem~\ref{thm: main}.

\begin{proof}[Proof of Theorem~\ref{thm: main}] 
Let $U = \pa_n w$ and $u = \pa_i w$, with $w$ as defined in \eqref{eqn: tilde u 0}.
Thanks to the rescalings at the beginning of this section, we have $\| \G\|_{C^{1,\a}}\leq 1.$ Up to possible further rescaling,  $U:=\pa_nw >0$ in $B_1$; see \cite{CSS}. Thanks to \eqref{eqn: equation}, \eqref{eqn: bound}, and Proposition~\ref{lem: derivative est C1a}, the remaining hypotheses of Proposition~\ref{prop: HOBH 1 pos} are satisfied up to multiplication by a universal constant. 
So, applying Proposition~\ref{prop: HOBH 1 pos} to $u$ and $U$, we obtain the existence of a polynomial $P$ of degree $1$ that, after a Taylor expansion of $d$, yields
\begin{equation*}
\bigg| \frac{\pa_i(v-\vphi)}{ \pa_n(v - \vphi) }-P \bigg| \leq C|x|^{1+\a}.
\end{equation*}
Up to translation and rotation, we may argue identically with $x_0 \in \G \cap B_{1/2}^*$ in place of the origin.
Hence, $\pa_i(v-\vphi)/ \pa_n(v - \vphi)\in C^{1,\a}(B_{1/2}^*)$. 
By a well-known argument (cf. \cite[Theorem 6.9]{PSU}), this implies that $\G\cap B^*_{1/4}\in C^{2,\a}$

Passing from $C^{k+2,\a}$ to $C^{k+3,\a}$ for $0 \leq k \leq m-4$ is identical; here Proposition~\ref{prop: deriv estimates} is used to show the hypotheses \eqref{eqn: grad x} and \eqref{eqn: grad r} of Proposition~\ref{prop: HOBH k pos} are satisfied. 
Applying Proposition~\ref{prop: HOBH k pos}, we find that there exists a polynomial $P$ of degree $k+2$ such that, after restricting to the hyperplane $\{y=0\}$, 
\[
\bigg| \frac{\pa_i(v-\vphi)}{ \pa_n(v - \vphi) }-P\bigg| \leq C|x|^{k+2+\a}.
\]
In turn, this implies that $\G \in C^{k+3, \a}$.
Arguing iteratively for $k = 0 , \dots, m-4$, the theorem is proved.
\end{proof}

%~~~~~~~~~~~~~~~~~~~~~~~~~~~~~~~~~~~~~~~~~~~~~~~~~~~~~~~%
\section{Appendix}

In this section, we prove Lemmas~\ref{lem: Approximations} and \ref{lem: C1a barrier}.

\begin{proof}[Proof of Lemma~\ref{lem: Approximations}]

The regularizations of $r$ and $U_a$ are constructed in the same way as the analogous regularizations in \cite{DS5}.
First, we smooth the signed distance function $d$ via convolution in $\lambda$-neighborhoods of $\G$.
Then, we define approximations $r_\l$ and $U_{a,\l}$ in geometrically shrinking annuli, and we patch them together in a smooth way.
The functions $r_\l$ and $U_{a,\l}$ here should not be confused with the rescalings of $r$ and $U_a$ defined in \eqref{eqn: rescaling}.

The functions $d, r$, and $U_a$ are locally Lipschitz in $B_1 \setminus \mathcal{P}$, and are therefore differentiable almost everywhere.
When we speak of their derivatives, we assume we are at a point of differentiability. \\

\noindent{\it Step 1: Construction and estimates for the function $d_\l$.}\\\\ 
Define the set $\mathcal{D}_\l := \{ x \in \R^n : |d| < 4\l \}$.
Let $\eta \in C^\infty_0(B^*_{1/50})$ be a positive, radially symmetric function that integrates to $1$, and
set
\[
d_\l := d*\eta_\l
\]
where $\eta_\l := \l^{-n} \eta(x/\l)$.
As in \cite{DS5}, the following estimates hold for $d_\l$ in $\mathcal{D}_{\l}$:
\begin{equation}
\label{eqn: bounds for d lambda}
|d_\l - d| \leq C\l^{1+\a}, 
\qquad 
|\na d_\l - \na d|\leq C\l^\a, 
\qquad\text{and}\qquad 
|D^2d_\l| \leq C\l^{\a -1}.
\end{equation}
In particular, the gradient estimate implies
\begin{equation}
\label{eqn: d lambda gradient}
|\na d_\l| = 1+ O(\l^\a).
\end{equation}

\noindent{\it Step 2: Construction and estimates for the function $r_\l$.}\\\\ 
Let 
\begin{equation}
\label{eqn: def R lambda}
\mathcal{R}_\l := \{ \l/2 < r < 4\l\} \subset \mathcal{D}_\l \qquad\text{and}\qquad r_\l:= (d_\l^2 + y^2)^{1/2} \quad\text{in } \mathcal{R}_\l.
\end{equation}
The following estimates hold in $\mathcal{R}_\l$:
\begin{equation}
\label{eqn: bounds for r lambda}
| r_\l - r| \leq C\l^{1+\a}, \quad |\na r_\l - \na r | \leq C\l^\a, \quad | D^2 r_\l| \leq C\l^{-1}, \quad\text{and}\quad \bigg|\Delta r_\l - \frac{1}{r} \bigg| \leq C\l^{\a-1}.	
\end{equation}
Consequently, we have that
\begin{equation}
\label{eqn: useful r lambda bounds}
\bigg|\frac{r_\l}{r} -1 \bigg| \leq C\l^\a, 
\qquad
\bigg| \frac{1}{r_\l} - \frac{1}{r} \bigg| \leq C\l^{\a -1},
\end{equation}
and
\begin{equation}
\label{eqn: r lambda gradient}
|\na r_\l | = 1 + O(\l^\a).
\end{equation}
All of these estimates were shown in \cite{DS5}, so we do not reprove them here.
Furthermore, we find that
\begin{equation}
\label{eqn: bounds for r lambda 2}
\bigg|\La r_\l -\frac{2(1-s)|y|^a}{r}\bigg| \leq C|y|^a\l^{\a-1}.
\end{equation}
To show \eqref{eqn: bounds for r lambda 2}, we express $\La(r_\l^2)$ in two different ways:
\[
2r_\l \La r_\l + 2 |y|^a |\na r_\l|^2\, =\,\La (r_\l^2) 
\, = \, 2|y|^a d_\l \Delta_x d_\l + 2 |y|^a |\na d_\l|^2 + 4(1-s) |y|^a.
\]
Then, \eqref{eqn: r lambda gradient}, \eqref{eqn: d lambda gradient}, and the third bound in \eqref{eqn: bounds for d lambda} imply that 
\[
r_\l \La r_\l = |y|^a( -|\na r_\l|^2 + d_\l \Delta_x d_\l + |\na d_\l|^2 + 2(1-s)) = |y|^a(2(1-s)+O(\l^\a)).
\]
Hence, \eqref{eqn: bounds for r lambda 2} follows from \eqref{eqn: useful r lambda bounds}.\\

\noindent{\it Step 3: Construction and estimates for the function $U_{a,\l}$.}\\\\ 
We define
\[
U_{a,\l} := \Big( \frac{d_\l + r_\l}{2}\Big)^s \quad \text{in }\mathcal{R}_\l,
\]
with $\mathcal{R}_\l$ defined as in \eqref{eqn: def R lambda}. 
The following bounds hold for $U_{a,\l}$ in $\mathcal{R}_\l$:
\begin{equation}
\label{eqn: bounds on Ual}
\bigg| \frac{U_{a,\l}}{U_a} - 1\bigg| \leq C\l^a, 
\quad | \na_x U_{a,\l} - \na_x U_a | \leq C\l^{\a -1+s}, \quad\text{and}\quad
|\pa_y U_{a,\l} - \pa_y U_a| \leq \frac{C\l^{\a +s}}{|y|},
\end{equation}
as well as
\begin{equation}
\label{eqn: bounds on Ual 2}
|\La U_{a,\l} | \leq C |y|^{a} \l^{\a - 2 + s}.
\end{equation}
These estimates must be shown separately in the regions
\[
\mathcal{R}_\l^+ := \mathcal{R}_\l \cap \Big\{ d\geq -\frac{r}{2}\Big\} \qquad\text{and}\qquad \mathcal{R}_\l^{-} := \mathcal{R}_\l \cap \left\{ d<-\frac{r}{2}\right\}.
\]
\\
\noindent{\it Step 3a: Estimates in $\mathcal{R}_\l^+$.} In $\mathcal{R}_\l^+$, the functions $U_a$ and $U_{a,\l}$ are comparable to $\l^s$. 
Also,
\[
U_{a,\l} = U_a \Big( \frac{ r_\l + d_\l}{r+d} \Big)^s .
\]
From the established bounds on $r_\l$ and $d_\l$ (\eqref{eqn: bounds for d lambda} and \eqref{eqn: bounds for r lambda}), we have 
\[
\frac{r_\l +d_\l}{r+d}  =  1+O(\l^\a) \qquad\text{and}\qquad
\na\Big( \frac{r_\l + d_\l}{r+d}\Big) = O(\l^{\a-1}).
\]
Therefore, 
\[
\bigg| \frac{U_{a,\l} }{U_a} -1\bigg| = \bigg| \Big( \frac{r_\l + d_\l}{r+d}\Big)^s -1\bigg| = O(\l^\a) 
\qquad \text{and} \qquad \na U_{a,\l} = \na U_a +O(\l^{s-1+\a}),
\]
proving all three estimates in \eqref{eqn: bounds on Ual} in $\mathcal{R}_\l^+$.
(Notice that we have actually shown a stronger estimate for the $y$-derivative in $\mathcal{R}_\l^+$.)
To determine the bound on $L_a U_{a,\l}$, we compute 
\[
s\La\big(U_{a,\l}^{1/s}\big) = U_{a,\l}^{1/s-1}\La U_{a,\l} + |y|^a\Big(\frac{1-s}{s}\Big) U_{a,\l}^{1/s -2} |\na U_{a,\l}|^2.
\]
On the other hand,
\[
\La (U_{a,\l}^{1/s}) = \La \Big(\frac{r_\l + d_\l}{2}\Big) 
= \frac{(1-s)|y|^a}{r} + |y|^aO( \l^{\a-1}).
\]
Together these estimates imply that
\begin{align*}
U_{a,\l}^{1/s-1} \La U_{a,\l} &= s|y|^a \Big( -\Big(\frac{1-s}{s^2}\Big) U_{a,\l}^{1/s -2}|\na U_{a,\l}|^2 + \frac{2(1-s)}{r} + O(\l^{\a-1})\Big) = |y|^a   O(\l^{\a-1}),
\end{align*}
where the second equality follows by \eqref{eqn: useful r lambda bounds} and using that
\[
|\na U_{a,\l}|^2 = s^2 \frac{U_{a,\l}^{2-1/s}}{r_\l} + O(\l^{\a-2 + 2s}) \quad\text{in } \mathcal{R}_\l^+.
\]
Multiplying by $U_{a,\l}^{1-1/s}$, we obtain \eqref{eqn: bounds on Ual 2}.
\\\\
\noindent{\it Step 3b: Estimates in $\mathcal{R}_\l^-$.}
In $\mathcal{R}_\l^-$, the functions $U_a$ and $U_{a, \l}$ are comparable to $|y|^{2s}\l^{-s}$. 
Indeed, $\frac{3}{2} r < r-d < 2r$ and $r+d = y^2/(r-d)$.
Thus, from \eqref{eqn: bounds for d lambda} and \eqref{eqn: bounds for r lambda}, we observe that
\[
\frac{r - d}{r_\l - d_\l }  =  1+O(\l^\a) \qquad\text{and}\qquad
\na\bigg( \frac{r - d}{r_\l - d_\l}\bigg) = O(\l^{\a-1}).
\]
As a consequence,
\[
U_{a,\l}  = U_a \Big(\frac{r-d}{r_\l - d_\l}\Big)^s = U_a (1+ O(\l^\a)).
\]
To see the $x$-gradient estimate in \eqref{eqn: bounds on Ual}, since $|\nab_x U_a| = sU_a/r \leq C\l^{s-1}$, we compute 
\[
\na_x U_{a,\l} = \na_x U_a + O(\l^{\a -1 + s}).
\]
Similarly, using that $|\partial_y U_a| = s|y|^{-1}U_a (r-d)/r \leq C|y|^{-1}\l^s$, we find that
\[
\partial_y U_{a,\l} = \partial_y U_a(1 + O(\l^{\a})) + U_aO(\l^{\a-1}) = \partial_y U_a + |y|^{-1}O(\l^{\a + s}).
\]
This proves \eqref{eqn: bounds on Ual}.
Finally, we compute $\La U_{a,\l}$ directly:
\begin{align*}
2^s\La U_{a,\l} &= \ \La (|y|^{2s} (r_\l - d_\l)^{-s} ) 
% &= 2s \pa_y( y|y|^{-1}(r_\l - d_\l )^{-s}) +  \DIV(|y| \na (r_\l - d_\l)^{-s} ) \\
= (1+2s)\frac{y}{|y|} \pa_y( (r_\l -d_\l)^{-s})  + |y| \Delta( (r_\l - d_\l)^{-s}) \\
&=  -\frac{s(1+2s)|y|}{r_\l(r_\l - d_\l)^{s+1}} 
-\frac{s|y| \Delta(r_\l -d_\l) }{(r_\l - d_\l)^{s+1}}+\frac{s(s+1)|y| |\na (r_\l - d_\l)|^2}{(r_\l - d_\l)^{s+2}}.
\end{align*}
Noting that $|\na (r_\l - d_\l)|^2  =  r_\l^{-2} (2r_\l (r_\l - d_\l) + O(\l^{2+\a}))$, recalling \eqref{eqn: bounds for r lambda} and \eqref{eqn: useful r lambda bounds}, and simplifying, we see that \eqref{eqn: bounds on Ual 2} holds.
\\\\
\noindent{\it Step 4: Construction and estimates for $r_*$ and $U_{a,*}.$}\\\\ 
The functions $r_*$ and $U_{a, *}$ are constructed by letting $\l_k = 4^{-k}$ and smoothly interpolating between $r_{\l_k}$ and $U_{a,\l}$ respectively with
\begin{align*}
r_* := 
\begin{cases}
r_{\l_k} &\text{in } \mathcal{R}_{\l_k} \cap\{ r\leq 2\l_k\} \\
r_{4\l_k} &\text{in } \mathcal{R}_{\l_k} \cap \{ r> 3\l_k\} 
\end{cases} 
\qquad\text{and}\qquad 
U_{a,*} := 
\begin{cases}
U_{a,\l_k} &\text{in } \mathcal{R}_{\l_k} \cap\{ r\leq 2\l_k\}\\
U_{a, 4\l_k} &\text{in } \mathcal{R}_{\l_k} \cap \{ r> 3\l_k\}
\end{cases}
\end{align*}
with $r_*$ and $U_{a,*}$ smooth in the intermediate region. 
More specifically, this is accomplished by defining
\[ 
\psi(t) := 
\begin{cases}
1 &\text{if } t \leq 2+\frac{1}{4}\\
0 &\text{if } t \geq 2+\frac{3}{4}
\end{cases}
\]
that is smooth for $2+\frac{1}{4}<t<2+\frac{3}{4}$, 
letting $\Psi := \psi( r_\l/\l )$, and setting
\[
r_* := \Psi r_\l + (1-\Psi )r_{4\l} \qquad\text{and}\qquad U_{a,*} :=\Psi U_{a,\l} + (1-\Psi )U_{a,4\l}.
\]

We use the estimates in \eqref{eqn: bounds for r lambda} to show the following bounds on $r_*$ in $\mathcal{R}_\l$:

\begin{equation}
\label{eqn: bounds for r star}
\begin{split}
| r_* - r| \leq C\l^{1+\a}, 
\qquad 
&|\na r_* - \na r | \leq C\l^\a, \\
|\pa_y r_* - \pa_y r| \leq C|y|^a r^{s-1}U_a \l^\a,
\qquad\text{and}\qquad
&\bigg|\La r_* -\frac{2(1-s)|y|^a}{r}\bigg| \leq C|y|^a\l^{\a-1}.
\end{split}
\end{equation}
Indeed, the first estimate follows from \eqref{eqn: bounds for r lambda} and since $0\leq \Psi \leq 1$. 
Next, keeping \eqref{eqn: r lambda gradient} in mind, the following estimates hold for $\Psi$:
\begin{equation}
\label{eqn: bounds on Psi}
% \begin{split}	
|\na \Psi| 
% &= \l^{-1}|h'(\l^{-1}r_\l)\na r_\l | 
\leq C\l^{-1}, \qquad
|D^2\Psi|  
\leq C\l^{-2}, \qquad\text{and}\qquad
|\pa_y \Psi| 
\leq C |y| \l^{-2}.
% \end{split}
\end{equation}
The remaining three inequalities in \eqref{eqn: bounds for r star} follow from \eqref{eqn: bounds on Psi}, the established estimates on $r_\l$, and an explicit computation.

Next using \eqref{eqn: bounds on Ual} and \eqref{eqn: bounds on Ual 2}, we show the following hold for $U_{a,*}$ in $\mathcal{R}_\l$:
\begin{equation}
\label{eqn: bounds on Ua star} 
\begin{split}
\bigg| \frac{U_{a,*}}{U_a} - 1\bigg| \leq C\l^\a,
\qquad 
\bigg|\frac{|\na U_{a,*}|}{|\na U_a |} - 1\bigg| \leq C\l^\a, 
\qquad\text{and}\qquad
|\La U_{a,*} | \leq C |y|^{a} \l^{\a - 2 + s}.
\end{split}
\end{equation}
The first inequality follows trivially.
Since $U_a/r \leq C|\nab U_a|$, we find that $\nab U_{a, \l} = \nab U_a(1 + O(\l^\a))$.
Hence, the second inequality in \eqref{eqn: bounds on Ua star} holds utilizing the first inequalities in \eqref{eqn: bounds on Ual} and \eqref{eqn: bounds on Psi}.
Finally, from \eqref{eqn: bounds on Ual} and \eqref{eqn: bounds on Psi} and an simple computation, one justifies the third estimate in \eqref{eqn: bounds on Ua star}.

Given \eqref{eqn: bounds for r star} and \eqref{eqn: bounds on Ua star}, the lemma follows.
\end{proof}

We now prove Lemma~\ref{lem: C1a barrier}. 

\begin{proof}[Proof of Lemma~\ref{lem: C1a barrier}]
We construct upper and lower barriers.
Define the upper barrier
\[ 
v_+ := U_{a,*} - U_{a,*}^{\beta}
\]
for some $\beta > 1$ that will be chosen.
Note that $v_+ \geq 0 $ on $\pa B_1 \cup \mathcal{P}$ since $\beta > 1$.
From Lemma~\ref{lem: Approximations}, we have that
\[
|\La U_{a,*}| \leq C_*\vep |y|^a r^{\a -2 + s}, \qquad
|U_{a,*}|  \leq  C_*r^s, \qquad\text{and}\qquad | \na U_{a, *}|^2 \geq c\frac{U_{a,*}^{2-1/s}}{r}.
\]
The third inequality holds provided that $0 < \vep \leq 1/2C_*$.
Indeed, 
\[
|\nab U_{a,*}| \geq (1-C_*\vep)|\nab U_a| = (1-C_*\vep)s\frac{U_{a}^{1-1/2s}}{r^{1/2}} \geq c\frac{U_{a,*}^{1-1/2s}}{r^{1/2}}.
\]
Now, observe that
\[
\begin{split}
\La v_+ 
&= \La U_{a,*} - \beta U_{a,*}^{\beta -1} \La U_{a,*} - |y|^a \beta(\beta-1)U_{a,*}^{\beta -2 }|\nab U_{a,*}|^2 \\
&\leq |y|^a(C_*\vep r^{\a -2 + s} - cr^{s\beta - 2} )
\end{split}
\]
if $\beta - 1/s < 0$.
Setting $\beta = 1 + \a/s$, we see that $\beta - 1/s < 0$ (recall that $\a \in (0,1-s)$), and so choosing $\vep > 0$ smaller if necessary, we deduce that 
\[
\La v_+ \leq  -c|y|^ar^{\a -2 + s}.
\]
We take $v_-:=-v_+$ as a lower barrier. 
The maximum principle ensures that 
\[
|u| \leq C |v_\pm| \leq C U_{a, *} \leq C U_a.
\]
\end{proof}

%~~~THE BIBLIOGRAPHY~~~%

\end{document}